\documentclass[final,12pt,authoryear]{elsarticle}
\usepackage{amssymb,amsthm,amsmath,graphicx}
\usepackage{amsfonts}
\usepackage{epsfig,latexsym,graphicx}
\usepackage[caption=false]{subfig}
\newtheorem{theorem}{Theorem}[section]

\newtheorem{corollary}[theorem]{Corollary}
\newtheorem{remark}{Remark}[section]

\theoremstyle{plain}

\biboptions{comma,round}
\begin{document}

\begin{frontmatter}

\title{On the Robustness of a Divergence based Test of Simple Statistical Hypotheses}
\author{Abhik Ghosh\fnref{label2}}
\ead{abhianik@gmail.com}
\author{Ayanendranath Basu\fnref{label1}}
\ead{ayanbasu@isical.ac.in}
\address{Indian Statistical Institute, Kolkata, India}
\author{Leandro Pardo}
\address{Complutense University, Madrid, Spain}
\ead{lpardo@mat.ucm.es}
\fntext[label2]{This is part of the Ph.D. research work of the first author
which is ongoing at the Indian Statistical Institute}
\fntext[label1]{Corresponding Author: Interdisciplinary Statistical Research Unit,\\
Indian Statistical Institute, \\
203 B. T. Road, Kolkata 700 108, India.\\
Phone: +91 33 2575 2806, 
Fax : +91 33 2577 3104.}
%  \fntext[label3]{203 B. T. Road, Kolkata 700 108, India}

\begin{abstract}
The most popular hypothesis testing procedure, the likelihood ratio test, is known to be highly non-robust
in many real situations. \cite{Basu/etc:2013a} provided an alternative robust procedure of hypothesis testing 
based on the density power divergence;
however, although the robustness properties of the latter test were intuitively argued for by the authors together 
with extensive empirical substantiation of the same, 
no theoretical robustness properties were presented in that work.
In the present paper we will consider a more general class of tests which forms a superfamily of the procedures 
described by \cite{Basu/etc:2013a}. This superfamily derives from the class of $S$-divergences 
recently proposed by \cite{Ghosh/etc:2013a}. In this context we theoretically prove several 
robustness results of the new class of tests and illustrate them in the normal model.
All the theoretical robustness properties of the \cite{Basu/etc:2013a} proposal follows as special cases 
of our results.
\end{abstract}


\end{frontmatter}

%\newpage

\section{Introduction}

Hypothesis testing is a very important component of statistical inference; 
it helps us to systematically explore the veracity of an unsubstantiated claim on the basis of observed 
data in a real life experiment. 
The philosophy of the statistical hypothesis testing was mainly developed in the early decades 
of twentieth Century by Fisher and Neyman and Pearson  
\citep{Fisher:1925b,Fisher:1935,Neyman/Pearson:1928, Neyman/Pearson:1933a, Neyman/Pearson:1933b}; 
since then the theory has evolved in many directions through the contributions of several later researchers. 
Yet the classical likelihood ratio test (LRT) proposed by \cite{Neyman/Pearson:1928}
and formalized later by \cite{Wilks:1938} still remains our canonical hypothesis testing tool, 
which is used widely by the practitioners in all 
scenarios of human endeavor. This test can be easily performed for most statistical models 
and satisfies several asymptotic optimality criteria. 
However, as in the case of maximum likelihood estimator (MLE) in the estimation context, 
the LRT also has serious robustness problems under misspecification of models and/or presence of outliers.

There have been several attempts to develop robust tests of hypotheses
having optimal properties similar to the LRT. 
The LRT can be seen to be a special case of a large class of testing procedures, 
known as the ``disparity difference tests" which have strong robustness properties \cite[e.g.~][]{Basu/etc:2011}. 
However, these tests involve kernel density estimators of the true density 
for continuous models and hence include all the complications associated with kernel smoothing. 
Motivated by the success of the minimum density power divergence estimators, 
which require no kernel smoothing, as an alternative to the class of minimum disparity estimators,
\cite{Basu/etc:2013a} developed a class of tests using the density power divergence. 
This work was based on \cite{Basu/etc:1998}, which had developed the density power divergence measure,
as well as the minimum density power divergence estimator (MDPDE).
Consider the problem of testing a simple null hypothesis based on the random sample $X_1, \ldots, X_n$. 
Let  $\mathcal{F} = \{f_\theta : \theta \in \Theta \subset \mathbb{R}^p\}$ represent the 
parametric family of densities; suppose also that the true data generating density belongs to this model family. 
Let $\theta_0$ be any fixed point in the parameter space $\Theta$, 
which we believe to be the true value of the parameter.  
Based on the observed sample, we want to test for the simple hypotheses 
\begin{equation}
 H_0 : \theta = \theta_0 ~~~ \mbox{against} ~~~~ H_1 : \theta \ne \theta_0.
 \label{EQ:7simple_hyp}
\end{equation}
Note that in general there is no uniformly most powerful test for this testing problem. 
\cite{Basu/etc:2013a} proposed to test this hypothesis
by utilizing the minimum possible value of the density power divergence (DPD) measure between 
the data and a model density. Although they have empirically demonstrated some of strong robustness properties 
of the DPD based test, their paper had no concrete theoretical results on the robustness of the proposed tests.

The present paper will focus on developing the theoretical robustness properties of the DPD based test. 
However, instead of doing it simply for the DPD alone, we will prove all our results for a more general class of 
test statistics that contains the DPD based tests as a special case. 
This general class of test statistics will be based on the recently developed  family of $S$-divergences 
\citep{Ghosh/etc:2013a} that contains both the PD (power divergence; Cressie and Read, 1984) and 
the DPD measure as its subfamilies. 
For all the asymptotic results throughout the paper 
we need to assume some standard conditions of asymptotic inference, 
given by Assumptions A, B, C and D of \citet[][p. 429]{Lehmann:1983}. 
In the rest of this paper, we will refer to these conditions simply as 
the ``{\it Lehmann conditions}". 	
Similarly, we will also assume the conditions D1--D5 of \citet[][p.~304]{Basu/etc:2011} at the model, 
which we will refer to as the ``{\it Basu et al.~conditions}".
Both set of conditions, with a brief description of their implications and significances 
are provided in \ref{App:conditions}.	
There is some overlap among the conditions, but taken together they exhaust the necessary technicalities.

To keep a clear focus in our discussion and to relate them directly to the empirical findings of 
\cite{Basu/etc:2013a}, we will restrict ourselves to a simple null hypothesis in this paper.
However, at the end of the paper, we will briefly indicate how the results can be extended
to the case of composite null hypotheses.

The rest of the paper is organized as follows: In Section \ref{SEC:7sample1_simpleTest}
we define the proposed test statistic based on the general family of $S$-divergence measures 
and provide its asymptotic properties. Several robustness measures of this test are derived in 
Section \ref{SEC:7sample1_simpleTest_robust}. Section \ref{SEC:7example_simple} presents a
numerical illustration of all the theoretical results derived in this paper through the problem 
of testing for the normal mean with a known variance.
In Section \ref{SEC:choice_tuning} we present some general remarks
integrating the theoretical results and numerical findings presented in this paper;
in this section we also  give some guidance on the choice of appropriate tuning parameter for the proposed test.
Section \ref{SEC:composite} briefly indicates the possible generalization 
to the case of the composite null hypotheses.
Some concluding remarks are presented in Section \ref{SEC:conclusion}.

%\newpage
%%%%%%%%%%%%%%%%%%%%%%%%%%%%%%%%%%%%%%%%%%%%%%
\section{General Test Statistics based on the $S$-Divergence}
\label{SEC:7sample1_simpleTest}
  
The $S$-divergence measure, recently proposed by \cite{Ghosh/etc:2013a}, 
is a general family of divergence measures (between two density functions) 
containing several popular divergences including
the power divergence (PD) family of Cressie and Read (1984) and 
the density power divergence (DPD) family of \cite{Basu/etc:2013a}. 
The $S$-divergence between the densities $g$ and $f$ is defined 
in terms of two parameters $\gamma \in [0, 1]$ and $\lambda \in \mathbb{R}$ as
\begin{equation}
S_{(\gamma, \lambda)}(g,f) =  \frac{1}{A} ~ \int ~ f^{1+\gamma}  -   \frac{1+\gamma}{A B} ~ 
\int ~~ f^{B} g^{A}  + \frac{1}{B} ~ \int ~~ g^{1+\gamma}, ~~~~~~ A \ne 0,~B \ne 0,
\label{EQ:S_div_gen}
\end{equation}
%%%%%%%%%%%%%%%%%%%%%%%%%%%%%%%%%%%%%%%%%%%%%%%%%%%%%%%%%%%%%%%%%%%%%%%%%%%%%%%%
where $A = 1+\lambda (1-\gamma)$ and  $B = \gamma - \lambda (1-\gamma)$.
For the cases $A=0$ or $B=0$, the corresponding $S$-divergence measure is defined by 
the respective continuous limits of the expression in Equation (\ref{EQ:S_div_gen}), which yield 
%%%%%%%%%%%%%%%%%%%%%%%%%%%%%%%%%%%%%%%%%%%%%%%%%%%%%%%%%%%%%%%%%%%%%%%%%%%%%%%%
\begin{equation}
S_{(\gamma,\lambda : A = 0)}(g,f) = \lim_{A \rightarrow 0} ~ S_{(\gamma, \lambda)}(g,f) 
=  \int f^{1+\gamma} \log\left(\frac{f}{g}\right) 
- \int \frac{(f^{1+\gamma} - g^{1+\gamma})}{{1+\gamma}},
\end{equation}
%%%%%%%%%%%%%%%%%%%%%%%%%%%%%%%%%%%%%%%%%%%%%%%%%%%%%%%%%%%%%%%%%%%%%%%%%%%%%%%%
and
%%%%%%%%%%%%%%%%%%%%%%%%%%%%%%%%%%%%%%%%%%%%%%%%%%%%%%%%%%%%%%%%%%%%%%%%%%%%%%%%
\begin{equation}
S_{(\gamma,\lambda : B = 0)}(g,f) = \lim_{B \rightarrow 0} ~ S_{(\gamma, \lambda)}(g,f) 
=  \int g^{1+\gamma} \log\left(\frac{g}{f}\right) 
- \int \frac{(g^{1+\gamma} - f^{1+\gamma})}{{1+\gamma}}.
\end{equation}
Interestingly, note that at $\gamma=1$ the $S$-divergence measure becomes independent of 
the parameter $\lambda$ and coincides with the squared $L_2$-distance. On the other hand, 
at $\gamma=0$ it reduces to the PD family of Cressie and Read (1984) with parameter $\lambda$.
Further, note that the $S$-divergence family also contains 
the DPD measure with parameter $\gamma = \beta$,
given by  
\begin{equation}
d_{\beta}(g,f) = \left\{\begin{array}{l c l}
\int  f_\theta^{1+\beta} - \frac{1+\beta}{\beta} \int f_\theta^\beta g 
+ \frac{1}{\beta} \int g^{1+\beta},  & \mbox{ if } & \beta > 0, \\\\
\int g \log(g/f_\theta), & \mbox{ if } & \beta = 0,
\end{array}\right.
\end{equation}
as one of its special cases when $\lambda=0$.
Therefore, it is natural to reconstruct the DPD based test statistics
proposed by \cite{Basu/etc:2013a} using this general family of $S$-divergences.

Consider the problem of testing the simple null hypothesis as described in Equation (\ref{EQ:7simple_hyp}) 
under the parametric set-up of Section 1. 
We define the general test statistic based on the $S$-divergence with parameter $\gamma$ and $\lambda$ as 
\begin{eqnarray}
\xi_n^{\gamma,\lambda}({\hat{\theta}_\beta}, {\theta_0}) 
= 2 n S_{(\gamma,\lambda)}(f_{\hat{\theta}_\beta}, f_{\theta_0}),
\label{EQ:SDT}
\end{eqnarray}
where $\hat{\theta}_\beta$ is the minimum DPD estimator (MDPDE) of $\theta$ obtained by minimizing the DPD 
with tuning parameter $\beta$.
The intuitive rationale behind construction of this test statistics is as follows: 
when the model is correctly specified and the assumed null hypothesis $H_0$ is correct, 
$f_{\theta_0}$ is the true data generating density and so it can be tested by considering the magnitude of the   
$S$-divergence  measure between $f_{\theta_0}$ and $f_{\hat{\theta}}$ for any consistent estimator $\hat{\theta}$ 
of $\theta$ based on the observed sample. 
Since the divergence employed in the definition of the statistic in Equation (\ref{EQ:SDT}) 
is the $S$-divergence, the ideal choice for $\hat{\theta}$ should be the minimum $S$-divergence estimator. 
The choice of $\hat{\theta}_\beta$, the MDPDE of $\theta$ corresponding to tuning parameter $\beta$, 
is preferred by us since the only subfamily in the $S$-divergence family 
that does not require the use of kernel density estimator is the DPD family.
Note that, by putting $\lambda=0$ this general test statistic coincides with the DPD based 
test statistic of \cite{Basu/etc:2013a}, whereas with the choice $\lambda =\gamma=\beta = 0$ it 
becomes equivalent to the LRT.

In the rest of the paper we refer to the test statistic $\xi_n^{\gamma,\lambda}({\hat{\theta}_\beta}, {\theta_0}) $ 
in the singular, although in effect we are taking a family of test statistics that vary over the tuning parameters.
We will prove our results for a generic statistic, where the asymptotic distribution is a function
 of $\gamma$, $\lambda$ and $\beta$.

We first prove some asymptotic properties of this general $S$-divergence based test 
(which we will refer to as the SDT) 
to obtain the critical points and power approximation of the test. 
For these, we require that the minimum DPD estimator used in constructing 
the test statistics is $n^{1/2}$-consistent and asymptotically normal, 
which hold under the Basu et al.~conditions \citep{Basu/etc:2011}.

%%%%%%%%%%%%%%%%%%%%%%%%%%%%%%%%%%%%%%%%%%%%%%
%\subsection{Asymptotic Properties}
Then, the asymptotic variance-covariance matrix of the MDPDE $\hat{\theta}_\beta$ is 
$$\Sigma_\beta(\theta) = J_\beta(\theta)^{-1} V_\beta(\theta) J_\beta(\theta)^{-1},$$ 
where
$$J_\beta(\theta) = \int u_\theta u_\theta^T f_\theta^{1+\beta},$$
and 
$$V_\beta(\theta) = \int u_\theta u_\theta^T f_\theta^{1+2\beta} 
- \left(\int u_\theta f_\theta^{1+\beta}\right)\left(\int u_\theta f_\theta^{1+\beta}\right)^T.$$ 
We can now derive the asymptotic null distribution of the SDT as given in the following theorem.

\bigskip
%---------------------------------------------------------------------
\begin{theorem}
	Suppose the model density satisfies the Lehmann and Basu et al.~conditions. 
The asymptotic distribution of the test statistic 	
$\xi_n^{\gamma,\lambda}({\hat{\theta}_\beta}, {\theta_0})$, under the null hypothesis 
$H_0 : \theta = \theta_0$, coincides with the distribution of 
$\sum_{i=1}^r ~  \zeta_i^{\gamma, \beta}(\theta_0)Z_i^2,$
where $Z_1, \ldots,Z_r$ are independent standard normal variables, 
$\zeta_1^{\gamma, \beta}(\theta_0), \ldots, \zeta_r^{\gamma, \beta}(\theta_0)$ 
are the nonzero eigenvalues of $A_\gamma(\theta_0)\Sigma_\beta(\theta_0)$
with 
$$
A_\gamma(\theta_0) = \nabla^2 S_{(\gamma,\lambda)}(f_\theta, f_{\theta_0})|_{\theta = \theta_0} =
 \left( (1+\gamma) \int f_{\theta_0}^{\gamma -1} \frac{\partial f_{\theta_0}}{\partial \theta_i} 
\frac{\partial f_{\theta_0}}{\partial \theta_j} \right)_{i,j=1,\cdots,p}
$$
and
$r = rank(V_\beta(\theta_0)J_\beta^{-1}(\theta_0)A_\gamma(\theta_0)J_\beta^{-1}(\theta_0)V_\beta(\theta_0)).$
%\qed
\label{THM:7asymp_null_one}
\end{theorem}

%The proof of above theorem is similar to that of the DPD based test. 
%Further, it can be seen that replacing $\hat{\theta}_\beta$ by the minimum 
%$S$-divergence  estimator with parameter $\beta$ and any $\lambda$
% the test statistic still have the same null asymptotic distribution. 
%This is because of the first order asymptotic equivalence of all the minimum $S$-divergence estimator 
%corresponding to any tuning parameter $\beta$ and $\lambda$ with the MDPDE having 
%tunning parameter as the same $\beta$ \citep{Ghosh/etc:2013a}.
%Further, note that the above null distribution of the general SDT is independent of the parameter $\lambda$ 
%and hence it is exactly the same as that obtained by \cite{Basu/etc:2013a} for the DPD based test. 
%So, for this general case also, we can find the critical point of the test statistic in the same way, 
%as per the Remark 3 of \cite{Basu/etc:2013a}. 

The proof of the above theorem is a routine extension of the proof of 
the DPD based test provided by \cite{Basu/etc:2013a} and is omitted. 
Note that the minimum $S$-divergence estimator corresponding to parameter $\beta$ and $\lambda$ 
is first order equivalent to the minimum DPD estimator with parameter $\beta$ 
in the sense that they have the same asymptotic distribution \citep{Ghosh/etc:2013a}. 
So, if one were to replace the minimum DPD estimator $\hat{\theta}_\beta$ in (\ref{EQ:SDT}) 
with the minimum $S$-divergence estimator corresponding to $\beta$ and $\lambda$, 
one would still get the same asymptotic null distribution as in Theorem \ref{THM:7asymp_null_one};
this null distribution is independent of $\lambda$. 
This allows us to study the class of tests in (\ref{EQ:SDT}) 
in a more general setting than what is observed at face value. 
We could have made the set up fully general by actually substituting $\hat{\theta}_\beta$ 
with the minimum $S$-divergence estimator at $\beta$ and $\lambda$ in Equation (\ref{EQ:SDT}). 
But the general $S$-divergence estimator involves the construction of 
a non-parametric density estimator and inherits its associated problems, 
and so the value addition due to the generality might be offset by its cost. 
Instead we stick to the formulation in (\ref{EQ:SDT}), 
which gives us the same asymptotic distribution as the most general case, 
but also gives us a robust set of procedures without getting into 
the issue of non-parametric kernel density estimation. 
It also allows us to study the theoretical properties of the tests of \cite{Basu/etc:2013a} 
as a special case of our formulation.
One could also find the asymptotic critical values of the test in (\ref{EQ:SDT}) 
on the basis of the suggestion in Remark 3 of \cite{Basu/etc:2013a}.

Our next theorem presents a power approximation of the above SDT, 
which can help us to obtain the minimum sample size required in order to achieve 
some pre-specified power of the test. 
\begin{theorem}
\label{THM:7aprox_power_one}
	Suppose the model density satisfies the Lehmann and Basu et al.~conditions. 
An approximation to the power function of the test statistic 
$\xi_n^{\gamma,\lambda}({\hat{\theta}_\beta}, {\theta_0})$ for testing 
the null hypothesis in Equation (\ref{EQ:7simple_hyp}) at the significance level $\alpha$ is given by
\begin{eqnarray}
\pi_{n,\alpha}^{\beta,\gamma,\lambda} (\theta^*) 
= 1 - \Phi_n \left( \frac{\sqrt{n}}{\sigma_{\beta,\gamma, \lambda}(\theta^*)} 
\left(\frac{t_\alpha^{\beta,\gamma}}{2 n} - S_{(\gamma,\lambda)}(f_{\theta^*}, f_{\theta_0})\right)\right), 
~~~ \theta^* \neq \theta_0,
\end{eqnarray}
where $\Phi_n$ tends uniformly to the standard normal distribution function $\Phi$, 
$t_\alpha^{\beta,\gamma}$ is the $(1-\alpha)^{\rm th}$ quantile of the asymptotic distribution of
$\xi_n^{\gamma,\lambda}({\hat{\theta}_\beta}, {\theta_0})$ and 
$\sigma_{\beta,\gamma,\lambda}^2(\theta) = 
M_{\gamma,\lambda}(\theta)^T\Sigma_\beta(\theta)M_{\gamma,\lambda}(\theta)$
with 
$$
M_{\gamma,\lambda}(\theta) = \nabla S_{(\gamma,\lambda)}(f_\theta , f_{\theta_0}) =
\frac{1+\alpha}{B} \left[\int f_\theta^{1+\alpha} u_\theta - \int f_{\theta_0}^B f_\theta^A u_\theta\right].
$$
Here $A=1+\lambda(1-\gamma)$ and $B=\gamma-\lambda(1-\gamma)$, 
as defined in the statement of the $S$-divergence.
\qed
\end{theorem}

\bigskip
\begin{corollary}
For any $\theta^* \neq \theta_0$, the probability of rejecting the null hypothesis $H_0$ 
at any fixed significance level $\alpha > 0 $ with the rejection rule 
$\xi_n^{\gamma,\lambda}({\hat{\theta}_\beta}, {\theta_0}) >  t_\alpha^{\beta,\gamma}$ 
tends to $1$ as $n \rightarrow \infty$. Thus the test statistic $\xi_n^{\gamma,\lambda}({\hat{\theta}_\beta}, {\theta_0}) $ is consistent.
\qed
\end{corollary}

%\bigskip

%%%%%%%%%%%%%%%%%%%%%%%%%%%%%%%%%%%%%%%%%%%%%% % % % % % % % % % % % % % % % % % % % % 
\section{Robustness Properties of the SDT}
\label{SEC:7sample1_simpleTest_robust}

Now let us consider the robustness properties of the general test statistic 
based on the $S$-divergence family (the SDT). 
We will see that the robustness of the proposed test statistic originates from 
the robustness of the minimum DPD estimator used in the test statistic and 
hence the proposed test will be seen to be highly robust for $\beta >0$.
As a particular case, all the results to be covered in this section will also hold 
for the DPD based test developed in \cite{Basu/etc:2013a}; this will fill the gap of the absence 
of any theoretical robustness results in their paper.

%-------------------------------------------------------------------
\subsection{Influence Function of the Test}

Let us first consider Hampel's influence function of the test 
\citep{Rousseeuw/Ronchetti:1979,Rousseeuw/Ronchetti:1981, Hampel/etc:1986}. 
Ignoring the multiplier $2n$ in our test statistic 
we define the SDT functional as
$$
T_{\gamma,\lambda}^{(1)}(G) = S_{(\gamma,\lambda)}(f_{T_\beta(G)},f_{\theta_0}),
$$ 
where $T_\beta(G)$ is the minimum DPD functional defined in \cite{Basu/etc:2011}. 
Now consider the contaminated distribution $G_\epsilon = (1-\epsilon) G + \epsilon \wedge_y$ 
with $\epsilon$ being the contamination proportion and $\wedge_y$ being the degenerate distribution
with all its mass at the contamination point $y$.
Then Hampel's first-order influence function of the SDT functional is given by 
$$
IF(y; T_{\gamma,\lambda}^{(1)}, G) 
= \left.\frac{\partial}{\partial\epsilon}T_{\gamma,\lambda}^{(1)}(G_\epsilon) \right|_{\epsilon=0} 
%= \left.M_{\gamma,\lambda}(T_\beta(G))^T 
%\frac{\partial}{\partial\epsilon}T_\beta(G_\epsilon)\right|_{\epsilon=0} 
= M_{\gamma,\lambda}(T_\beta(G))^T IF(y; T_\beta, G),
$$
where $IF(y; T_\beta, G) = \frac{\partial}{\partial\epsilon}T_\beta(G_\epsilon) \big|_{\epsilon=0}$ 
is the influence function of the minimum DPD functional $T_\beta$ and 
$M_{\gamma,\lambda}(T_{\beta}(G))
 = \frac{\partial}{\partial\theta}S_{(\gamma,\lambda)}(f_{\theta},f_{\theta_0})\big|_{\theta=T_\beta(G)}$. 
In general the Influence function of a test is evaluated at the null distribution $G=F_{\theta_0}$. 
However, by the Fisher consistency property of the functional $T_\beta(\cdot)$, 
we get $T_\beta(F_{\theta_0}) = \theta_0$ and $M_{\gamma,\lambda}(\theta_0)=0$ 
so that the Hampel's first-order influence function of our test statistic is zero at the null.

Therefore, we need to consider higher order influence functions of the proposed SDT. 
The second order influence function of our test statistic can be seen to have the form
\begin{eqnarray}
IF_2(y; T_{\gamma,\lambda}^{(1)}, G) &=& 
\frac{\partial^2}{\partial^2\epsilon}T_{\gamma,\lambda}^{(1)}(G_\epsilon) \big|_{\epsilon=0} \nonumber \\
&=& M_{\gamma,\lambda}(T_\beta(G))^T 
\frac{\partial^2}{\partial\epsilon^2} T_\beta(G_\epsilon)\big|_{\epsilon=0} \nonumber\\
&&  + IF(y; T_\beta, G)^T \nabla^2 S_{(\gamma,\lambda)}(f_\theta,f_{\theta_0})\big|_{\theta=T_\beta(G)} 
IF(y; T_\beta, G). \nonumber
\end{eqnarray}
In particular, at the null distribution $G=F_{\theta_0}$, 
the second order influence function of the SDT becomes 
\begin{equation}
IF_2(y; T_{\gamma,\lambda}^{(1)}, F_{\theta_0}) 
= IF(y; T_\beta, F_{\theta_0})^T A_\gamma(\theta_0) IF(y; T_\beta, F_{\theta_0}).
\label{EQ:IF_SDT}
\end{equation}
Thus the influence function of the $S$-divergence based test at the null hypothesis is independent 
of the parameter $\lambda$ implying that the theoretical robustness of the proposed SDT will be independent 
of $\lambda$. Also, this influence function will be bounded if and only if the influence function of 
the minimum density power divergence functional is bounded. 
\cite{Basu/etc:1998} derived the influence function of the MDPDE 
which is bounded for all $\beta>0$ under most common parametric models;
however it is generally unbounded at $\beta=0$. Noting that the MDPDE with $\beta=0$ is indeed the MLE, 
it shows that the use of a non-robust estimator like the MLE leads to a non-robust overall 
test procedure as well.

%Using the form of influence function of the MDPDE, the influence function of the SDT becomes 
%$$
%IF_2(y; T_{\gamma,\lambda}^{(1)}, F_{\theta_0}) 
%= (1+\gamma) \left[u_{\theta_0}(y)f_{\theta_0}^\beta(y) - \int u_{\theta_0}f_{\theta_0}^{1+\beta}\right]^T 
%\frac{\int u_{\theta_0}^2f_{\theta_0}^{1+\gamma}}{\int u_{\theta_0}^2f_{\theta_0}^{1+\beta}} 
%\left[ u_{\theta_0}(y)f_{\theta_0}^\beta(y) - \int u_{\theta_0} f_{\theta_0}^{1+\beta} \right].
%$$
%For most of the parametric models, this influence function is bounded whenever $\beta>0$; 

%---------------------------------------------------------------------------------------------------
\subsection{Level and Power under contamination and the corresponding Influence Functions}

Next we consider the effect of contamination on the level and power of the proposed test 
which will give us a clearer picture about the robustness of the test. 
As the test is consistent, we study its power under the contiguous alternative hypotheses  
$\theta_n = \theta_0 + \frac{\Delta}{\sqrt{n}}$ with $\Delta$ being 
a vector of positive reals having the same dimension as the parameter vector and $\theta_n \in \Theta$. 
In order to explore the effect of contamination on the power and size of the test, 
we also need to consider some contamination over these contiguous alternatives. 
Following \cite{Hampel/etc:1986}, one must consider the contaminations such that 
their effect tends to zero as $\theta_n$ tends to $\theta_0$ at the same rate 
to avoid confusion between the null and alternative neighborhoods 
\cite[also see][]{Huber/Carol:1970, Heritier/Ronchetti:1994, Toma/Broniatowski:2010}. 
Therefore, we consider the contaminated distributions 
$$
F_{n,\epsilon,y}^L = \left(1-\frac{\epsilon}{\sqrt{n}}\right) F_{\theta_0} 
+ \frac{\epsilon}{\sqrt{n}} \wedge_y ~~~~~~~~~~~~~~ \mbox{ for level, }
$$
and
$$
F_{n,\epsilon,y}^P = \left(1-\frac{\epsilon}{\sqrt{n}}\right) F_{\theta_n} 
+ \frac{\epsilon}{\sqrt{n}} \wedge_y ~~~~~~~~~~~~~~ \mbox{ for power. }
$$
Then the level influence function (LIF) is given by
$$
LIF(y; T_{\gamma,\lambda}^{(1)}, F_{\theta_0}) = \lim_{n \rightarrow \infty} 
~ \frac{\partial}{\partial \epsilon} 
P_{F_{n,\epsilon,y}^L}(\xi_n^{\gamma,\lambda}({\hat{\theta}_\beta}, {\theta_0})> t_\alpha^{\beta,\gamma}) 
\big|_{\epsilon=0},
$$
and the power influence function (PIF) is given by 
$$
PIF(y; T_{\gamma,\lambda}^{(1)}, F_{\theta_0}) = \lim_{n \rightarrow \infty} 
~ \frac{\partial}{\partial \epsilon} 
P_{F_{n,\epsilon,y}^P}(\xi_n^{\gamma,\lambda}({\hat{\theta}_\beta}, {\theta_0})>t_\alpha^{\beta,\gamma}) 
\big|_{\epsilon=0}.
$$
We begin with the derivation of the asymptotic power under contaminated distributions
which is presented in the  following theorem.

\bigskip
%------------------------------------------
\begin{theorem}
	Assume that the Lehmann and Basu et al.~conditions hold for the model density.
	Then for any $\Delta \in \mathbb{R}^p$ and $\epsilon \geq 0$, we have the following:
	\begin{itemize}
		\item[(i)] The asymptotic distribution of the $S$-divergence based test statistics 
		$\xi_n^{\gamma,\lambda}({\hat{\theta}_\beta}, {\theta_0})$ under $F_{n,\epsilon,y}^P$ is the same 
		as that of the quadratic form $W^T A_\gamma(\theta_0)W$, where $W$ follows a $p$-variate normal distribution 
		with mean 
		$$\widetilde{\Delta} = \left[ \Delta + \epsilon IF(y;U_\beta,F_{\theta_0})\right]$$ 
		and variance-covariance matrix $\Sigma_\beta(\theta_0)$. 
		Equivalently, this distribution is the same as that of 
		$\sum_{i=1}^r ~  \zeta_i^{\gamma, \beta}(\theta_0)\chi_{1,\delta_i}^2,$
		where $\zeta_1^{\gamma, \beta}(\theta_0), \cdots, \zeta_r^{\gamma, \beta}(\theta_0)$ 
		are the $r$ nonzero eigenvalues of $A_\gamma(\theta_0)\Sigma_\beta(\theta_0)$ 
		as in Theorem \ref{THM:7asymp_null_one} and $\chi_{1,\delta_1}^2, \ldots, \chi_{1,\delta_r}^2$ 
		are independent non-central chi-square variables having degree of freedom one and non-centrality parameters 
		$\delta_1, \ldots, \delta_r$ respectively with $\delta_i = \mu_i^2$ and 
		$\mu = (\mu_1,\ldots,\mu_p)^T =  P_{\beta,\gamma}(\theta_0)\Sigma_\beta^{-1/2}(\theta_0)\widetilde{\Delta}$ and 
		$P_{\beta,\gamma}(\theta_0)$ is the matrix of normalized eigenvectors of $A_\gamma(\theta_0)\Sigma_\beta(\theta_0)$.
		
		\item[(ii)] The asymptotic power of the proposed SDT under contaminated distribution 
		$F_{n,\epsilon,y}^P$ is given by 
		\begin{eqnarray}
		Power(\Delta, \epsilon) &=& \lim_{n \rightarrow \infty} ~ 
		P_{F_{n,\epsilon,y}^P}(\xi_n^{\gamma,\lambda}({\hat{\theta}_\beta}, {\theta_0}) >  t_\alpha^{\beta,\gamma}) 
		\nonumber \\
		&=& \sum\limits_{v=0}^{\infty} ~ C_v^{\gamma, \beta}(\theta_0, \widetilde{\Delta}) 
		P\left(\chi_{r+2v}^2 > \frac{t_\alpha^{\beta,\gamma}}{\zeta_{(1)}^{\gamma, \beta}(\theta_0)}\right),
		\label{EQ:7asymp_power_cont_null}
		\end{eqnarray}
		where $\chi_p^2$ denotes a chi-square random variable with $p$ degrees of freedom,  
		$\zeta_{(1)}^{\gamma, \beta}(\theta_0)$ is the minimum of $\zeta_{i}^{\gamma, \beta}(\theta_0)$s 
		for $i=1, \ldots,r$  and
		\begin{eqnarray}
		C_v^{\gamma, \beta}(\theta_0, \widetilde{\Delta}) = \frac{1}{v!}\left(\prod\limits_{j=1}^{r}
		\frac{\zeta_{(1)}^{\gamma, \beta}(\theta_0)}{\zeta_{j}^{\gamma,\beta}(\theta_0)}\right)^{1/2}  \cdot
		e^{-\frac{\delta}{2}} E(\hat{Q}^v), \nonumber
		\end{eqnarray}
		with $\delta = \mu^T\mu = \sum\limits_{j=1}^{r}\delta_j$ and 
		$$
		\hat{Q} = \frac{1}{2} \sum\limits_{j=1}^{r} 
		\left[\left(1 - \frac{\zeta_{(1)}^{\gamma, \beta}(\theta_0)}{\zeta_{j}^{\gamma,\beta}(\theta_0)}\right)^{1/2}Z_j 
		+ \mu_j \left(\frac{\zeta_{(1)}^{\gamma,\beta}(\theta_0)}{\zeta_{j}^{\gamma,\beta}(\theta_0)}\right)^{1/2}\right]^2,
		$$
		for $r$ independent standard normal random variables $Z_1 \ldots, Z_r$.
	\end{itemize}
	\label{THM:7asymp_power_one}
\end{theorem}
\begin{proof}	
	To prove part (i), let us denote $\theta_n^* = U_\beta(F_{n,\epsilon,y}^P)$.
	We consider the second order Taylor series expansion of $S_{(\gamma,\lambda)}(f_\theta, f_{\theta_0})$ 
	around $\theta = \theta_n^*$ at $\theta = \hat{\theta}_\beta $ as,
	\begin{eqnarray}
	S_{(\gamma,\lambda)}(f_{\hat{\theta}_\beta}, f_{\theta_0}) 
	&=& S_{(\gamma,\lambda)}(f_{\theta_n^*}, f_{\theta_0}) 
	+ M_{\gamma,\lambda}(\theta_n^*)^T (\hat{\theta}_\beta - \theta_n^*) \nonumber\\
	&& ~~ + \frac{1}{2}(\hat{\theta}_\beta - \theta_n^*)^TA_\gamma(\theta_n^*)(\hat{\theta}_\beta - \theta_n^*) 
	+ o(||\hat{\theta}_\beta - \theta_n^*||^2). \nonumber
	\end{eqnarray}
	Now from the asymptotic distribution of the MDPDE and the consistency of $\theta_n^*$ we know that 
	$\sqrt{n} (\hat{\theta}_\beta - \theta_n^*) \displaystyle\mathop{\rightarrow}^\mathcal{D}  
	N(0, \Sigma_\beta(\theta_0) ),$
	and $ A_\gamma (\theta_n^*) \displaystyle\mathop{\rightarrow}^\mathcal{P} A_\gamma (\theta_0)$ as $n\rightarrow \infty$.
	%from the proof of Theorem 2.1
	%we get that under the probability $F_{n,\epsilon,y}^P$, 
	%$$
	%\sqrt{n} (\hat{\theta}_\beta - \theta_n^*)^T A_\gamma (\theta_n^*) \sqrt{n} (\hat{\theta}_\beta - \theta_n^*) 
	%\mathop{\rightarrow}^\mathcal{D} F_{H_0}.
	%$$
	Further using the Taylor series expansion of $M_{\gamma,\lambda}(\theta)$ 
	around $\theta = \theta_0$ at $\theta = \theta_n^*$, we get 
	$$
	M_{\gamma,\lambda}( \theta_n^*) - M_{\gamma,\lambda}(\theta_0) 
	= \frac{1}{\sqrt{n}}  A_\gamma(\theta_0)\Delta + \frac{\epsilon}{\sqrt{n}} IF(y;M_{\gamma,\lambda},F_{\theta_0}) 
	+ o\left(n^{-1/2}\right)
	$$
	But $M_{\gamma,\lambda}(\theta_0)=0$ and  
	$IF(y;M_{\gamma,\lambda},F_{\theta_0})= A_\gamma(\theta_0) IF(y; U_\beta,F_{\theta_0})$ 
	so that we get 
	\begin{eqnarray}
	\sqrt{n} M_{\gamma,\lambda}( \theta_n^*) 
	%&=&   A_\gamma(\theta_0)\Delta + \epsilon A_\gamma(\theta_0) IF(y,U_\beta,F_{\theta_0}) 
	%+ o\left(\frac{1}{\sqrt{n}}\right) \nonumber\\
	&=&   A_\gamma(\theta_0)\widetilde{\Delta}+ o\left(1\right), \nonumber
	\end{eqnarray}
	with $\widetilde{\Delta} = \left[ \Delta + \epsilon IF(y;U_\beta,F_{\theta_0})\right]$. 
	Again using the second order Taylor series expansion of $S_{(\gamma,\lambda)}(f_{\theta}, f_{\theta_0})$ 
	around $\theta=\theta_0$ at $\theta=\theta_n^*$, we get
	\begin{eqnarray}
	&& S_{(\gamma,\lambda)}(f_{\theta_n^*}, f_{\theta_0}) - S_{(\gamma,\lambda)}(f_{\theta_0}, f_{\theta_0}) 
	\nonumber\\
	&& = \frac{\epsilon}{\sqrt{n}} IF(y; T_{\gamma,\lambda}^{(1)},F_{\theta_0}) 
	+ \frac{1}{\sqrt{n}} \Delta^T M_{\gamma,\lambda}(\theta_0) 
	+ o\left(n^{-1/2}\right) \times M_{\gamma,\lambda}(\theta_0) 
	\nonumber \\ && 
	+ \frac{\epsilon^2}{2n} IF_2(y; T_{\gamma,\lambda}^{(1)}, F_{\theta_0})  
	+ \frac{1}{2n} \Delta^T A_\gamma(\theta_0)\Delta  
	+ \frac{\epsilon}{n} \Delta^T IF(y;M_{\gamma,\lambda},F_{\theta_0})  + o\left(\frac{1}{n}\right). 
	\nonumber\\
	&&~~
	\label{EQ:7number1}
	\end{eqnarray}
	But, $S_{(\gamma,\lambda)}(f_{\theta_0}, f_{\theta_0})=0$, 
	$IF(y; T_{\gamma,\lambda}^{(1)},F_{\theta_0})=0$ and $IF_2(y; T_{\gamma,\lambda}^{(1)}, F_{\theta_0})$
	is given in Equation (\ref{EQ:IF_SDT}).
%	$$ IF_2(y; T_{\gamma,\lambda}^{(1)}, F_{\theta_0}) 
%	= IF(y; U_\beta, F_{\theta_0})^T A_\gamma(\theta_0) IF(y; U_\beta, F_{\theta_0}).$$ 
	Thus Equation (\ref{EQ:7number1}) simplifies to 
	\begin{eqnarray}
	2n S_{(\gamma,\lambda)}(f_{\theta_n^*}, f_{\theta_0}) &=& \Delta^T  A_\gamma(\theta_0)\Delta 
	+ \epsilon^2 IF(y; U_\beta, F_{\theta_0}) A_\gamma(\theta_0)^T IF(y; U_\beta, F_{\theta_0}) \nonumber\\ 
	&& + 2\epsilon \Delta^T A_\gamma(\theta_0) IF(y; U_\beta, F_{\theta_0}) 
	+ o\left(1\right). \nonumber
	\end{eqnarray}
	Now noting that $n \times o(||\hat{\theta}_\beta - \theta_n^*||^2) = o_p(1)$, we get 
	\begin{eqnarray}
	2n S_{(\gamma,\lambda)}(f_{\hat{\theta}_\beta}, f_{\theta_0}) 
	&=& \widetilde{\Delta}^T  A_\gamma(\theta_0)\widetilde{\Delta} 
	+  2 \widetilde{\Delta}^T A_\gamma(\theta_0) \sqrt{n}(\hat{\theta}_\beta - \theta_n^*) \nonumber\\
	&+&\sqrt{n}(\hat{\theta}_\beta-\theta_n^*)^TA_\gamma(\theta_n^*)\sqrt{n}(\hat{\theta}_\beta - \theta_n^*) 
	+ o_P(1) +  o\left(1\right)\nonumber\\
	%&& ~~~~ + o_P(1) +  o\left(1\right)\nonumber \\
	&=& \left[\widetilde{\Delta} + \sqrt{n}(\hat{\theta}_\beta-\theta_n^*)\right]^T  A_\gamma(\theta_0)
	\left[\widetilde{\Delta} + \sqrt{n}(\hat{\theta}_\beta-\theta_n^*)\right] \nonumber\\
	&& ~~~~~~~~~~~~~~~~~~~~~~~~~~~~~~~~+ o_P(1) +  o\left(1\right).\nonumber 
	\end{eqnarray}
	Thus under the probability $F_{n,\epsilon,y}^P$, the asymptotic distribution of the SDT statistic
	$\xi_n^{\gamma,\lambda}({\hat{\theta}_\beta}, {\theta_0}) 
	= 2n S_{(\gamma,\lambda)}(f_{\hat{\theta}_\beta}, f_{\theta_0})$ 
	is the same as the distribution of $(\widetilde{\Delta}+W_0)^TA_\gamma(\theta_0)(\widetilde{\Delta}+W_0)$, 
	where $W_0$ is a random variable following $N(0, \Sigma_\beta(\theta_0) )$ distribution. 
	Hence the first statement of part (i) follows by taking $W=(\widetilde{\Delta}+W_0)$.

%	Next considering the spectral decomposition of the matrix 
%	$\Sigma_\beta(\theta_0)^{1/2}A_\gamma(\theta_0)\Sigma_\beta(\theta_0)^{1/2}$, we get 
	Now employing a spectral decomposition we get
	$$
	\Sigma_\beta(\theta_0)^{1/2}A_\gamma(\theta_0)\Sigma_\beta(\theta_0)^{1/2}
	=P_{\beta,\gamma}(\theta_0)^T \Gamma_r(\theta_0) P_{\beta,\gamma}(\theta_0),
	$$
	where $P_{\beta,\gamma}(\theta_0)$ is as defined in the statement of the theorem and $\Gamma_r(\theta_0)$ is 
	the diagonal matrix having diagonal entries as the eigenvalues  of $A_\gamma(\theta_0)\Sigma_\beta(\theta_0)$.
	Now
	\begin{eqnarray}
	&& W^T A_\gamma(\theta_0)W =  W^T\Sigma_\beta(\theta_0)^{-1/2} 
	\left[\Sigma_\beta(\theta_0)^{1/2}A_\gamma(\theta_0)\Sigma_\beta(\theta_0)^{1/2}\right]
	\Sigma_\beta(\theta_0)^{-1/2}W \nonumber\\
	&& ~~~= \left(W_0 + \widetilde{\Delta}\right)^T\Sigma_\beta(\theta_0)^{-1/2} 
	\left[P_{\beta,\gamma}(\theta_0)^T \Gamma_r(\theta_0) P_{\beta,\gamma}(\theta_0)\right]
	\Sigma_\beta(\theta_0)^{-1/2}\left(W_0 + \widetilde{\Delta}\right) \nonumber\\
	&& ~~~= (W^* + \mu)^T\Gamma_r(\theta_0) (W^*+\mu), \nonumber
	\end{eqnarray} 
	where $W^* = (W_1^*, \ldots, W_p^*)= P_{\beta,\gamma}(\theta_0)\Sigma_\beta(\theta_0)^{-1/2}W_0$ 
	follows an $N(0, I_p)$ distribution and $\mu$ is as defined in the statement of the theorem. Thus, 
	$$
	W^T A_\gamma(\theta_0)W  = \sum_{i=1}^r ~  \zeta_i^{\gamma, \beta}(\theta_0)(W_i+\mu_i)^2,
	$$
	completing the proof of second statement of part (i).
	
	\bigskip
	Part (ii) follows from part (i) using the series expansion of the distribution function of 
	a linear combination of independent non-central chi-squares in terms of central chi-square distribution functions
	as derived in \cite{Kotz/etc:1967b}.
	
	%\begin{eqnarray}
	%P(\Delta, \epsilon) &=& \lim_{n \rightarrow \infty} ~ 
	%P_{F_{n,\epsilon,y}^P}(\xi_n^{\gamma,\lambda}({\hat{\theta}_\beta}, {\theta_0}) >  t_\alpha^{\beta,\gamma}) 
	%\nonumber \\
	%&=& P(W_1+W_2> t_\alpha^{\beta,\gamma}) = \int P(W_1> t_\alpha^{\beta,\gamma}-x) dF_{H_0}(x)\nonumber \\
	%&=& 1 - \int \Phi\left(\frac{t_\alpha^{\beta,\gamma} - y 
	%- \widetilde{\Delta}^T A_\gamma(\theta_0)\widetilde{\Delta}}{2\sqrt{\widetilde{\Delta}^T A_\gamma(\theta_0) 
	%\Sigma_\beta(\theta_0) A_\gamma(\theta_0) \widetilde{\Delta}}}\right) 
	%dF_{H_0}(y) \nonumber
	%\end{eqnarray}
	%This completes the proof. 
\end{proof}
%-----------------------------------------------------------------------------

\bigskip
Interestingly note that the asymptotic power under contiguous alternatives and contiguous contamination 
is independent of the parameter $\lambda$; thus asymptotically the power robustness of the 
SDT will be independent of $\lambda$. 
Further, substituting $\Delta=0$ or $\epsilon=0$ in the above theorem, we will get several 
important cases; these are presented in the following corollaries.

\bigskip
%----------------------------------------------------------
\begin{corollary}
	\label{COR:7cont_power_null}
	For $\Delta=\epsilon=0$, the distribution  $F_{n,\epsilon,y}^P$ coincides with the null distribution. 
	Putting $\Delta=\epsilon=0$ in Theorem \ref{THM:7asymp_power_one}, 
	we get $\widetilde{\Delta}=0$ so that $\delta_i =0$ for each $i$ and 
	the asymptotic distribution of the SDT statistics given in the previous theorem coincides 
	with its null distribution derived independently in Theorem \ref{THM:7asymp_null_one}.
	So part (ii) of Theorem \ref{THM:7asymp_power_one} gives an alternative approximation to 
	the critical values of the proposed SDT based on a suitable truncation  of the infinite series.
	Error bounds in such truncation can also be calculated from Remark \ref{REM:expansion_non-central_chi2} below 
	or from the results of \cite{Kotz/etc:1967a,Kotz/etc:1967b}.
	\qed
\end{corollary}

\medskip
%----------------------------------------------------------
\begin{corollary}
	\label{COR:7cont_power_one}
	Putting $\epsilon=0$ in Theorem \ref{THM:7asymp_power_one}, 
	we get the asymptotic distribution of the proposed SDT under the contiguous alternative hypotheses
	$H_{1,n}: \theta= \theta_n = \theta_0 + \frac{\Delta}{\sqrt{n}}$ as given in part (i) of 
	Theorem \ref{THM:7asymp_power_one} 	with $\widetilde{\Delta} = \Delta$. 
	Further, in this case part (ii) of the theorem yields the
	asymptotic power under contiguous alternative hypotheses given by  
	\begin{eqnarray}
	P_0 &=& Power(\Delta, \epsilon=0)  
	%\nonumber \\&=& 
	=\sum\limits_{v=0}^{\infty} ~ C_v^{\gamma, \beta}(\theta_0, {\Delta}) 
	P\left(\chi_{r+2v}^2 > \frac{t_\alpha^{\beta,\gamma}}{\zeta_{(1)}^{\gamma, \beta}(\theta_0)}\right). \nonumber
	\end{eqnarray}
	\qed
\end{corollary}

\medskip
%-------------------------------------------------------------------------------------
\begin{corollary}
	Putting $\Delta=0$ in Theorem \ref{THM:7asymp_null_one}, 
	we get the asymptotic distribution of the proposed SDT under the probability distribution $F_{n,\epsilon,y}^L$ 
	as given in part (i) of the theorem with $\widetilde{\Delta} = \epsilon IF(y;U_\beta,F_{\theta_0})$.
	Hence, the asymptotic level under contamination has the form 
	\begin{eqnarray}
	\alpha_\epsilon &=& Power(\Delta=0,\epsilon)  
	%\nonumber \\&=& 
	=\sum\limits_{v=0}^{\infty} ~ C_v^{\gamma, \beta}(\theta_0, \epsilon IF(y;U_\beta,F_{\theta_0})) 
	P\left(\chi_{r+2v}^2 > \frac{t_\alpha^{\beta,\gamma}}{\zeta_{(1)}^{\gamma, \beta}(\theta_0)}\right). \nonumber
	\end{eqnarray}
	\qed
	\label{COR:7asymp_level_one}
\end{corollary}
%--------------------------------------------------------------------------------

\bigskip
%-------------------------------------------------------------------------------------
\begin{remark}
	In the above Theorem \ref{THM:7asymp_null_one}, we have used a series expansion for the distribution of 
	the linear combination of independent non-central chi-square random variables in terms of the central chi-square
	distributions from the work of \cite{Kotz/etc:1967a}. There are in fact many more such expansions or  
	approximations for the distribution of these random variables available in the literature;
	for example see \cite{Press:1966}, \cite{Harville:1971} and 
	\cite{Liu/etc:2009} among many others. However, the particular expansion used here 
	is specifically useful for deriving the power influence function, as we will see next.
	
	Further, for any practical usage, we can use a finite truncated sum as an approximation 
	of the infinite series considered in (\ref{EQ:7asymp_power_cont_null}). For the truncated series 
	up to the $N$-th term, the error in approximation can be bounded by 
	\begin{eqnarray}
	e_N &=& \sum\limits_{v=N+1}^{\infty} ~ C_v^{\gamma, \beta}(\theta_0, \widetilde{\Delta}) 
	P\left(\chi_{r+2v}^2 > \frac{t_\alpha^{\beta,\gamma}}{\zeta_{(1)}^{\gamma, \beta}(\theta_0)}\right) 
	\nonumber\\
	&\leq& \sum\limits_{v=N+1}^{\infty} ~ C_v^{\gamma, \beta}(\theta_0, \widetilde{\Delta}) 
	= 1 - \sum\limits_{v=0}^{N} ~ C_v^{\gamma, \beta}(\theta_0, \widetilde{\Delta}). \nonumber
	\end{eqnarray}
	See \cite{Kotz/etc:1967a,Kotz/etc:1967b} for more accurate error bounds for such approximations.
	\qed
	\label{REM:expansion_non-central_chi2}
\end{remark}
%--------------------------------------------------------------------------------

\bigskip
Next we derive the power influence function of the proposed test statistics. 
Starting with the expression of $P(\Delta,\epsilon)$ as obtained in Theorem \ref{THM:7asymp_power_one}, 
we get the power influence function $PIF(\cdot)$ as
\begin{eqnarray}
PIF(y; T_{\gamma,\lambda}^{(1)}, F_{\theta_0}) 
&=& \left.\frac{\partial}{\partial\epsilon}  Power(\Delta,\epsilon) \right|_{\epsilon=0} \nonumber\\
&=&\sum\limits_{v=0}^{\infty}~\left.\frac{\partial}{\partial\epsilon}C_v^{\gamma,\beta}(\theta_0,\widetilde{\Delta})
\right|_{\epsilon=0} P\left(\chi_{r+2v}^2 > \frac{t_\alpha^{\beta,\gamma}}{\zeta_{(1)}^{\gamma, \beta}(\theta_0)}\right).~~~~
\label{EQ:7PIF_0simpleTest1}
\end{eqnarray} 
Now, note that for each $v \geq 0$, the quantities $C_v^{\gamma, \beta}(\theta_0, \widetilde{\Delta})$
depend on $\epsilon$ only through their second argument 
$\widetilde{\Delta} = \left[ \Delta + \epsilon IF(y;U_\beta,F_{\theta_0})\right]$
and at $\epsilon=0$ we have $\widetilde{\Delta} = \Delta$. Consider a Taylor series 
expansion of $C_v^{\gamma, \beta}(\theta_0, t)$ with respect to $t$ around $t=\Delta$
and evaluate it at $t=\widetilde{\Delta}$ to get
\begin{eqnarray}
&& C_v^{\gamma, \beta}(\theta_0, \widetilde{\Delta}) 
= C_v^{\gamma, \beta}(\theta_0, {\Delta}) + (\widetilde{\Delta} - \Delta) \cdot 
\left[\left.\frac{\partial}{\partial t}C_v^{\gamma,\beta}(\theta_0,t)^T\right|_{t=\Delta}\right]
+ o(||\widetilde{\Delta} - \Delta||)\nonumber\\
&&= C_v^{\gamma, \beta}(\theta_0, {\Delta}) + \epsilon IF(y;U_\beta,F_{\theta_0})^T \cdot 
\left[\left.\frac{\partial}{\partial t}C_v^{\gamma,\beta}(\theta_0,t)\right|_{t=\Delta}\right]\nonumber\\
&& ~~+ o(\epsilon IF(y;U_\beta,F_{\theta_0})).~~~~~
\label{EQ:7PIF_00simpleTest1}
\end{eqnarray}
Now differentiating it with respect to $\epsilon$ and evaluating at $\epsilon=0$, 
we get 
$$
\left.\frac{\partial}{\partial\epsilon}C_v^{\gamma,\beta}(\theta_0,\widetilde{\Delta})\right|_{\epsilon=0}
= IF(y;U_\beta,F_{\theta_0})^T \cdot
\left[\left.\frac{\partial}{\partial t}C_v^{\gamma,\beta}(\theta_0,t)\right|_{t=\Delta}\right],
$$
provided the influence function $IF(y;U_\beta,F_{\theta_0})$ is bounded.
Combining it with Equation (\ref{EQ:7PIF_0simpleTest1}), we finally get the required power influence function 
that is presented in the following theorem.

\bigskip
%--------------------------------------------------------------------------
\begin{theorem}\label{THM:7PIF_simpleTest1}
	Assume that the Lehmann and Basu et al.~conditions hold for the model density and 
	the influence function $IF(y; U_{\beta}, F_{\theta_0})$ of the minimum DPD estimator is bounded. 
	Then, for any $\Delta \in \mathbb{R}^p$, the power influence function of the proposed SDT is given by 
	
	\begin{eqnarray}
	&& PIF(y; T_{\gamma,\lambda}^{(1)}, F_{\theta_0}) 
	= \left.\frac{\partial}{\partial\epsilon}  Power(\Delta,\epsilon) \right|_{\epsilon=0} \nonumber\\
	&&= IF(y;U_\beta,F_{\theta_0})^T \left(\sum\limits_{v=0}^{\infty}~
	\left[\left.\frac{\partial}{\partial t}C_v^{\gamma,\beta}(\theta_0,t)\right|_{t=\Delta}\right] 
	P\left(\chi_{r+2v}^2 > \frac{t_\alpha^{\beta,\gamma}}{\zeta_{(1)}^{\gamma, \beta}(\theta_0)}\right)\right).
	\nonumber\\
	&& ~~
	\label{EQ:7PIF_simpleTest1}
	\end{eqnarray} 
	\qed
\end{theorem}
%-------------------------------------------------------------------
\bigskip
Note that the quantity $\left(\sum\limits_{v=0}^{\infty}~
\left[\left.\frac{\partial}{\partial t}C_v^{\gamma,\beta}(\theta_0,t)\right|_{t=\Delta}\right] 
P\left(\chi_{r+2v}^2 > \frac{t_\alpha^{\beta,\gamma}}{\zeta_{(1)}^{\gamma, \beta}(\theta_0)}\right)\right) = C^*$, 
say,  is independent of the parameter $\lambda$ and the contamination point $y$. 
Thus the above Theorem shows that the power influence function 
is bounded whenever the influence function of the MDPDE is bounded. 
But the latter condition is satisfied by most standard parametric models as shown
in \cite{Basu/etc:1998,Basu/etc:2011}. 
This implies the power robustness of the proposed test statistics and 
the extent of theoretical robustness is again independent of the parameter $\lambda$.

Further, combining Equations (\ref{EQ:7asymp_power_cont_null}) and (\ref{EQ:7PIF_00simpleTest1}),
we can get an interesting interpretation of the power influence function. 
Substituting the value of $C_v^{\gamma, \beta}(\theta_0, \widetilde{\Delta}) $ from (\ref{EQ:7PIF_00simpleTest1})
in to the Expression (\ref{EQ:7asymp_power_cont_null}) and using Corollary \ref{COR:7cont_power_one}, 
we get 
\begin{eqnarray}
Power(\Delta, \epsilon) &=& P_0 + \epsilon PIF(y; T_{\gamma,\lambda}^{(1)}, F_{\theta_0}) + o(\epsilon^2),
\nonumber
\end{eqnarray}
whenever the influence function of the MDPDE is bounded. Thus, the power influence function 
gives a first order approximation for the change in asymptotic power under contiguous alternative hypotheses
caused by a contiguous contamination in data. This is in fact comparable with the interpretation 
of the influence function of any estimator as an indicator of the change in its bias 
under contaminated data.

Finally, we can derive the level influence function of the SDT
just by putting $\Delta=0$ in above Expression (\ref{EQ:7PIF_simpleTest1}). 
Thus we have
\begin{eqnarray}
&& LIF(y; T_{\gamma,\lambda}^{(1)}, F_{\theta_0}) 
= \left.\frac{\partial}{\partial\epsilon}  Power(\Delta=0,\epsilon) \right|_{\epsilon=0} \nonumber\\
&&~~= IF(y;U_\beta,F_{\theta_0})^T \left(\sum\limits_{v=0}^{\infty}~
\left[\left.\frac{\partial}{\partial t}C_v^{\gamma,\beta}(\theta_0,t)\right|_{t=0}\right] 
P\left(\chi_{r+2v}^2 > \frac{t_\alpha^{\beta,\gamma}}{\zeta_{(1)}^{\gamma, \beta}(\theta_0)}\right)\right).
~~~~~\nonumber
\label{EQ:7LIF_simpleTest1}
\end{eqnarray} 
Thus, whenever the influence function of the MDPDE is bounded, which is the case for all $\beta >0$, 
the asymptotic level of the proposed test statistics will be unaffected by the contiguous contamination. 

\bigskip
\begin{remark}\textbf{[The case of $p=r=1$]}\\
	All the results derived in this subsection including the power and level influence function 
	can be further simplified for the cases of univariate parameters with $p=r=1$.
	In this particular case, it follows from Theorem \ref{THM:7asymp_null_one} that the 
	asymptotic null distribution of the proposed SDT statistic is the same as 
	$\zeta_{(1)}^{\gamma, \beta}(\theta_0) Z_1$, where $Z_1$ follows a chi-square distribution with 
	one degree of freedom. Thus, the critical value for a level $\alpha$ SDT is given by 
	$t_\alpha^{\beta,\gamma}=\zeta_{(1)}^{\gamma, \beta}(\theta_0) \chi_{1,\alpha}^2$, 
	where $\chi_{1,\alpha}^2$ is upper $\alpha^{\rm th}$ quantile of the distribution of $Z_1$.
	
	Now, it follows from Theorem \ref{THM:7asymp_power_one} that the distribution of SDT statistic 
	under the contaminated contiguous alternatives $F_{n,\epsilon,y}^P$ is the same as that of 
	$\zeta_{(1)}^{\gamma, \beta}(\theta_0) W_{1,\delta}$, where $W_{1,\delta}$ follows a non-central 
	chi-square distribution with one degree of freedom and non-centrality parameter 
	$\delta = \widetilde{\Delta}^2\Sigma_\beta(\theta_0)^{-1}$. 
	Let $G_{\chi_1^2(\delta)}(\cdot)$ denote the distribution function of $W_{1,\delta}$.
	Then, we have
	\begin{eqnarray}
	Power(\Delta, \epsilon) &=& P\left(\zeta_{(1)}^{\gamma, \beta}(\theta_0) W_{1,\delta} 
	>  t_\alpha^{\beta,\gamma} = \zeta_{(1)}^{\gamma, \beta}(\theta_0) \chi_{1,\alpha}^2\right) 
	\nonumber \\
	&=&  1 - G_{\chi_1^2(\delta)}(\chi_{1,\alpha}^2) \nonumber\\
	&=& \sum\limits_{v=0}^{\infty} ~ C_v^{\gamma, \beta}(\theta_0, \widetilde{\Delta}) 
	P\left(\chi_{r+2v}^2 > \chi_{1,\alpha}^2\right),\nonumber
	\label{EQ:7asymp_power_cont_null_1}
	\end{eqnarray}
	where $C_v^{\gamma, \beta}(\theta_0, \widetilde{\Delta}) $ has now the simpler form given by
	$$
	C_v^{\gamma, \beta}(\theta_0, t) = \frac{e^{-\frac{t^2}{2}\Sigma_\beta(\theta_0)^{-1}} t^{2v}}{v! 2^v}
	\Sigma_\beta(\theta_0)^{-v}, ~~~~~~ v =0, 1, \ldots.~~
	$$
	Therefore, we can differentiate it with respect to $t$ to get, for each $v \geq 0$,
	$$
	\frac{\partial}{\partial t}C_v^{\gamma, \beta}(\theta_0, t) 
	= \frac{e^{-\frac{t^2}{2}\Sigma_\beta(\theta_0)^{-1}}\left(2v t^{2v-1} - t^{2v+1}\Sigma_\beta(\theta_0)^{-1}\right)
	}{v! 2^v}
	\Sigma_\beta(\theta_0)^{-v}.
	$$
	Then, following Theorem \ref{THM:7PIF_simpleTest1}, 
	the power influence function of the proposed SDT, in this case, is given by 
	\begin{eqnarray}
	&& PIF(y; T_{\gamma,\lambda}^{(1)}, F_{\theta_0}) \nonumber\\
	&& = IF(y;U_\beta,F_{\theta_0}) e^{-\frac{\Delta^2}{2\Sigma_\beta(\theta_0)}} \left[\sum\limits_{v=0}^{\infty}~
	\frac{P\left(\chi_{r+2v}^2 > \chi_{1,\alpha}^2 \right)}{v! 2^v \Sigma_\beta(\theta_0)^{v} }
	\left(2v \Delta^{2v-1} - \frac{\Delta^{2v+1}}{\Sigma_\beta(\theta_0)}\right)\right].
	~~~~~\nonumber
	\label{EQ:7PIF_1_simpleTest1}
	\end{eqnarray} 
	Further, it also follows that the level influence function of the proposed SDT in this case is in fact zero 
	whenever $IF(y;U_\beta,F_{\theta_0}) $ is bounded, 
	because here $\frac{\partial}{\partial t}C_v^{\gamma, \beta}(\theta_0, t) =0$ at $t=0$.
	Therefore the proposed SDT with $\beta>0$ for any one-dimensional parameter will always 
	be robust with respect to its size and power.
	\qed
	\label{REM:7SDT0_PIF_1}
\end{remark}

%========================================================================
\subsection{The Chi-Square Inflation Factor of the Test}
\label{SEC:Chi-square-inflation}

In this section we will further explore the robustness of the SDT in terms of 
the stability of its limiting distribution under general contamination. 
In contrast to the contiguous contamination considered in the preceding section, 
here we consider the fixed departure from our model density that is independent of the sample size $n$. 
Following \cite{Lindsay:1994}, let us consider the unknown true distribution to be $G$ that 
may or may not be in our model family and let the null hypothesis of interest be 
\begin{equation}
H_0 : T_\beta(G)=\theta_0.
\label{EQ:7null_hyp_func}
\end{equation}
Then our $S$-divergence test statistics for testing the above can be written as 
$\xi_{\gamma,\lambda}({\hat{\theta}_\beta}, {\theta_0}) 
= 2 n S_{(\gamma,\lambda)}(f_{\hat{\theta}_\beta}, f_{\theta_0}),$
where $\hat{\theta}_\beta = T_\beta(G_n)$ with $G_n$ being the empirical distribution function.
For simplicity, we restrict our attention to the case of scalar $\theta$ only; 
generalization to multivariate case can be made by a routine extension
although the interpretation is more difficult for the multivariate case.

%----------------------------------------------------------------------------
\begin{theorem}
Consider the above set-up with a scalar parameter $\theta$ and 
let $g$ denote the density function  corresponding to $G$.
Then under the null hypothesis in (\ref{EQ:7null_hyp_func}),
\begin{equation}
\xi_n^{\gamma,\lambda}({\hat{\theta}_\beta}, {\theta_0}) \mathop{\rightarrow}^\mathcal{D} c(g) \chi_1^2,
\label{EQ:7Chi_infl_fact}
\end{equation}
where $c(g) = A_\gamma(\theta_0) V_\beta(g) J_\beta^{-2}(g)$ with 
$A_\gamma(\theta_0) = (1+\gamma) \int u_{\theta_0}^2 f_{\theta_0}^{1+\gamma},$
$$
J_\beta(g) = \int u_{\theta_0}^2 f_{\theta_0}^{1+\beta} 
+ \int (i_{\theta_0} - \beta u_{\theta_0}^2)(g-f_{\theta_0})f_{\theta_0}^\beta,
$$ with $i_{\theta}= - \nabla u_\theta$ and 
$V_\beta(g) = \int u_{\theta_0}^2 f_{\theta_0}^{2\beta}g - \left(\int u_{\theta_0}f_{\theta_0}^\beta g \right)^2.$
\label{THM:7Chi_infl_fact}
\end{theorem}
\begin{proof}	
Following the same proof as in Theorem \ref{THM:7asymp_null_one},
we have that under the above null hypothesis, the asymptotic distribution of the test statistics 
$\xi_n^{\gamma,\lambda}({\hat{\theta}_\beta}, {\theta_0})$ is the same as that of 
$A_\gamma(\theta_0) \left[\sqrt{n} ({\hat{\theta}_\beta} - {\theta_0})\right]^2$. 
However since $T_\beta(G)=\theta_0$, it follows that under $g$, 
$\sqrt{n} ({\hat{\theta}_\beta} - {\theta_0})$ is asymptotically univariate normal with mean zero 
and variance $V_\beta(g) J_\beta^{-2}(g)$ so that the asymptotic distribution of 
$A_\gamma(\theta_0) \left[\sqrt{n} ({\hat{\theta}_\beta} - {\theta_0})\right]^2$ is nothing but $c(g) \chi_1^2$.
\end{proof}
%-------------------------------------------------------------------

As in \cite{Lindsay:1994} for the disparity difference test, 
here also we refer $c(g)$ as the {\it Chi-Square Inflation Factor} 
having  exactly the similar interpretations. 
However, the confidence interval of $T_{\gamma,\lambda}^{(1)}$ based on 
Theorem \ref{THM:7Chi_infl_fact} will only be correct whenever $c(g)= c(f_{\theta_0})$; 
it will be conservative provided $c(g)< c(f_{\theta_0})$ and liberal otherwise. 
Clearly, $c(g)$ is independent of the parameter $\lambda$ and depends on the true density $g$ only 
through the parameter $\beta$ of the MDPDE. 
Thus the stability of this SDT and corresponding confidence interval under the model misspecification 
solely depends on the robustness of the MDPDE used in the test statistics.

As an illustration, we study in detail, 
the properties of this chi-square inflation factor $c(g)$  for the particular case, 
$g = g_\epsilon = (1-\epsilon)f_{\theta_0} + \epsilon \wedge_y$. 
In this case, the value of the quantity $\frac{\partial}{\partial \epsilon}c(g_\epsilon) \big|_{\epsilon=0}$ 
gives us the extent of stability of the test statistics under small contaminations 
with similar interpretation as that of the influence function. 
Whenever its value is bounded in $y$, we get the stable test statistics under small contaminations and 
the magnitude of its supremum over all $y$ gives the extent of stability. 
The following theorem gives an explicit form of this quantity; 
the proof involves routine differentiation and is hence omitted.

\begin{theorem}
\label{THM:7Chi_infl_fact_slope}
Consider the above set-up with scalar parameter $\theta$. Then we have,
\begin{eqnarray}
&& \frac{\partial}{\partial \epsilon}c(g_\epsilon) \big|_{\epsilon=0} 
= (1+\gamma) \frac{M_\gamma}{M_\beta^3} \left[ M_\beta u_{\theta_0}^2(y) f_{\theta_0}^{2\beta}(y) 
- 2f_{\theta_0}^\beta(y) \right. \nonumber\\
&& ~~ \times \left. \left\{ N_\beta M_\beta u_{\theta_0}(y) (M_{2\beta} - N_\beta^2)(i_{\theta_0}(y)
-\beta u_{\theta_0}^2(y)) \right\} + \Upsilon_\beta (\theta_0) \right],
\end{eqnarray}
where $M_\beta = \int u_{\theta_0}^2 f_{\theta_0}^{1+\beta}$, 
$N_\beta = \int u_{\theta_0} f_{\theta_0}^{1+\beta}$ and 
$$
\Upsilon_\beta (\theta_0)= 2(M_{2\beta} - N_\beta^2)(\int i_{\theta_0}f_{\theta_0}^{1+\beta} - (1+\beta)M_\beta)
$$
are independent of the contamination point $y$. 
\qed
\end{theorem}
%

%========================================================================
\section{An Illustration: Testing Normal Mean with known variance}
\label{SEC:7example_simple}

Let us now consider the most popular but simple hypothesis testing problem 
regarding the mean $\theta$ of 
an univariate normal with known variance $\sigma^2$; 
so $f_\theta$ is the $N(\theta, \sigma^2)$ density with $\theta$ being the parameter of interest.
Based on a sample $X_1, \ldots, X_n$ of size $n$ from this population, 
we want to test the simple hypothesis 
$H_0: \theta=\theta_0$ against the omnibus alternative. Let $\hat{\theta}_\beta$ denote the MDPDE 
of $\theta$ under this set-up. Then, using the form of the normal density, 
the SDT has the simple form given by 
$$
\xi_n^{\gamma,\lambda}({\hat{\theta}_\beta}, {\theta_0})  = 2 n \kappa_\gamma \frac{1+\gamma}{AB} 
\left[1 - e^{-\frac{AB(\hat{\theta}_\beta -\theta_0)^2}{2(1+\gamma)\sigma^2}}\right], ~~ A, B \neq 0,
$$
where $\kappa_\gamma=(2\pi)^{-\frac{\gamma}{2}} \sigma^{-\gamma} (1+\gamma)^{-\frac{1}{2}}$.
Whenever one of $A$ or $B$ is zero, the test statistic has the limiting form given by 
$$
\xi_n^{\gamma,\lambda}({\hat{\theta}_\beta}, {\theta_0})  
=\frac{ n \kappa_\gamma }{\sigma^2}(\hat{\theta}_\beta -\theta_0)^2, ~~ A = 0~ \mbox{or} ~ B = 0.
$$
Further note that at $\gamma=\lambda=\beta=0$, we have $\kappa_0=1$, $B=0$ and 
$\hat{\theta}_0=\bar{X} = \hat{\theta}_{MLE}$, the MLE of $\theta$;
thus the SDT statistic becomes 
$$\xi_n^{0,0}({\hat{\theta}_0}, {\theta_0})  = \frac{n}{\sigma^2}(\bar{X} -\theta_0)^2,$$ 
which is the ordinary likelihood ratio test statistic for the problem under consideration.

Now, the asymptotic null distribution of the above SDT follows from Theorem \ref{THM:7asymp_null_one}; 
the normal density family can be seen to satisfy the required assumptions.
Recall from \cite{Basu/etc:2011} that the asymptotic distribution of 
$\sqrt{n}(\hat{\theta}_\beta - \theta_0)$ is asymptotically normal with mean $0$ and variance 
$\upsilon_\beta = \frac{(1+\beta)^3}{(1+2\beta)^{3/2}}\sigma^2$ under $H_0: \theta =\theta_0$.
Then the asymptotic null distribution of the SDT statistics 
$\xi_n^{\gamma,\lambda}({\hat{\theta}_\beta}, {\theta_0}) $ is the same as 
that of $\zeta_1^{\gamma,\beta} Z_1$, where $Z_1$ follows a $\chi^2$ distribution with one degree of freedom  
and $\zeta_1^{\gamma,\beta} = \frac{\kappa_\gamma\upsilon_\beta}{\sigma^2}$; 
in other words, 
$$
\frac{\xi_n^{\gamma,\lambda}({\hat{\theta}_\beta}, {\theta_0})}{\zeta_1^{\gamma, \beta}}  
\mathop{\rightarrow}^\mathcal{D} \chi^2_1, ~~~~~~\mbox{as }~~n\rightarrow\infty.
$$
In particular when $\gamma=\beta=0$, we have $\zeta_1^{\gamma,\beta}=1$ as expected.

Further, we can compute an approximation to the asymptotic power of the above test 
at any alternative point $\theta^*$ using Theorem \ref{THM:7aprox_power_one}. 
To determine the power in the spirit of the contiguous alternatives, we choose the alternatives 
$H_1:  \theta^* = \theta_0 + \frac{\Delta}{\sqrt{n}}$ 
and plot the corresponding approximate power in the Figure \ref{FIG:7Approx_power_comparison}
for $\beta=\gamma$, $\sigma=1$, $\theta_0=0$, $\Delta=\sqrt{10}$ 
and different sample sizes $n = 10, 50, 100, 300$. 
It is clear from the figures that the approximate power of the proposed SDT is almost the same 
for different values of $\lambda$ whenever the sample sizes are large enough or the values of the parameters 
$\gamma=\beta$ is close to one; note that the second case makes the the $S$-divergence measure 
to be almost independent of $\lambda$ (it is completely independent at $\gamma=1$). 
However, for the small sample sizes with smaller values of $\gamma=\beta$, 
the approximate powers depend on $\lambda$ significantly.

\begin{figure}
\centering
%------------------------------------------------------
\subfloat[Appr., $n=30$]{
\includegraphics[width=0.4\textwidth,height=0.19\textheight] {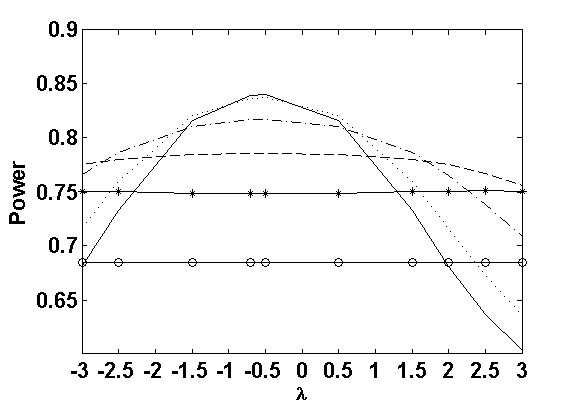}
\label{fig:apower_30}}
~ %--------------------------------------------------------------------
\subfloat[Sim., $n=30$]{
\includegraphics[width=0.4\textwidth,height=0.19\textheight] {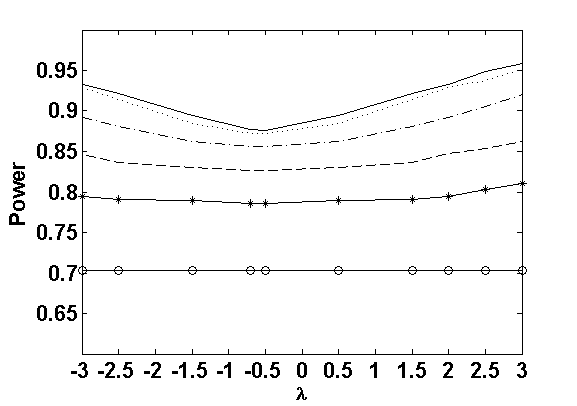}
\label{fig:Spower_30}}
\\ %--------------------------------------------------------------------
\subfloat[Appr., $n=50$]{
\includegraphics[width=0.4\textwidth,height=0.19\textheight] {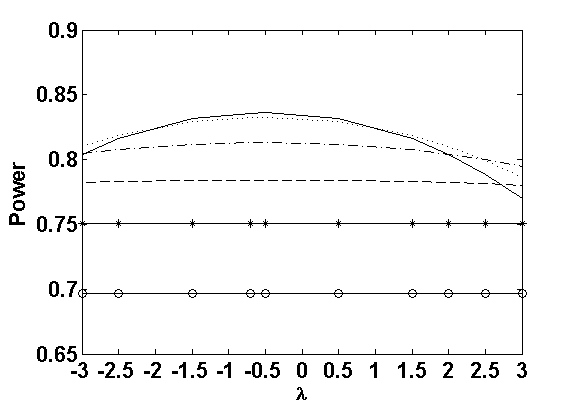}
\label{fig:apower_50}}
~ %--------------------------------------------------------------------
\subfloat[Sim., $n=50$]{
\includegraphics[width=0.4\textwidth,height=0.19\textheight] {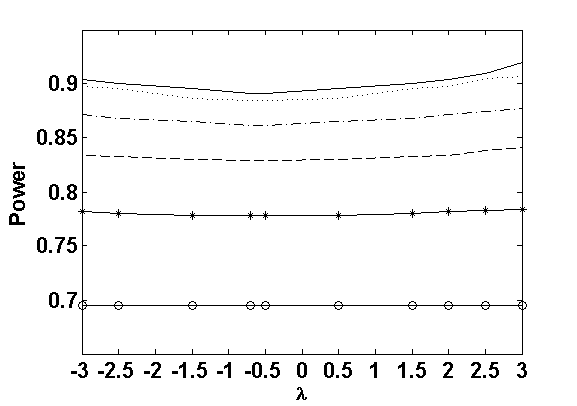}
\label{fig:Spower_50}}
\\ %--------------------------------------------------------------------
\subfloat[Appr., $n=100$]{
\includegraphics[width=0.4\textwidth,height=0.19\textheight] {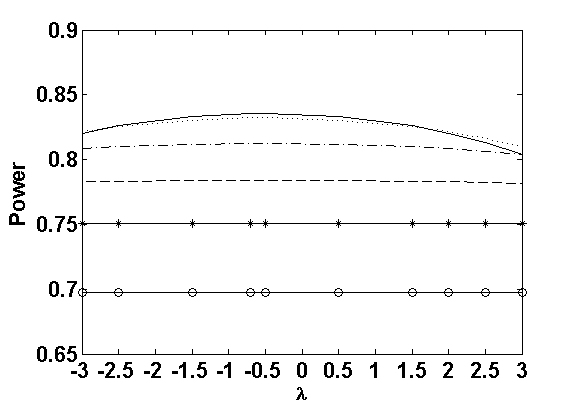}
\label{fig:apower_100}}
~ %--------------------------------------------------------------------
\subfloat[Sim., $n=100$]{
\includegraphics[width=0.4\textwidth,height=0.19\textheight] {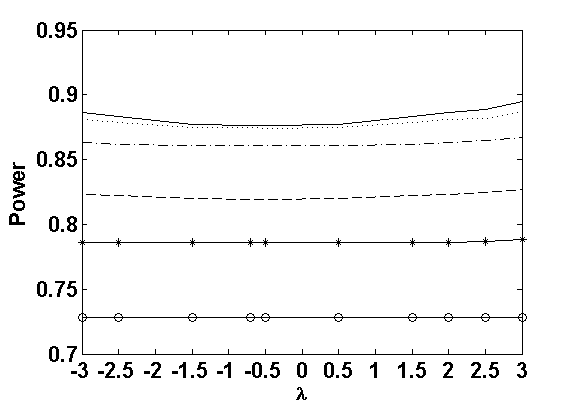}
\label{fig:Spower_100}}
\\ %--------------------------------------------------------------------
\subfloat[Appr., $n=300$]{
\includegraphics[width=0.4\textwidth,height=0.19\textheight] {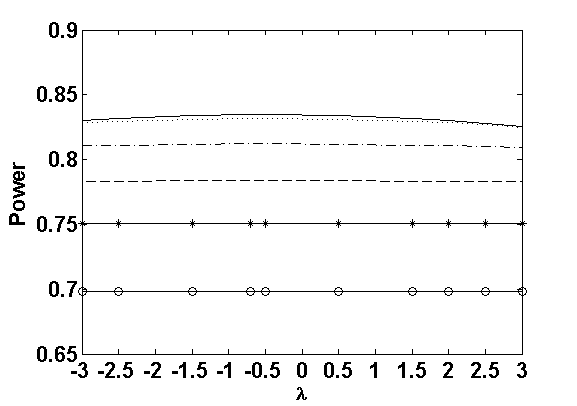}
\label{fig:apower_300}}
~%------------------------------------------------------\includegraphics[width=0.24\textwidth] {apower10.png}
\subfloat[Sim., $n=300$]{
\includegraphics[width=0.4\textwidth,height=0.19\textheight] {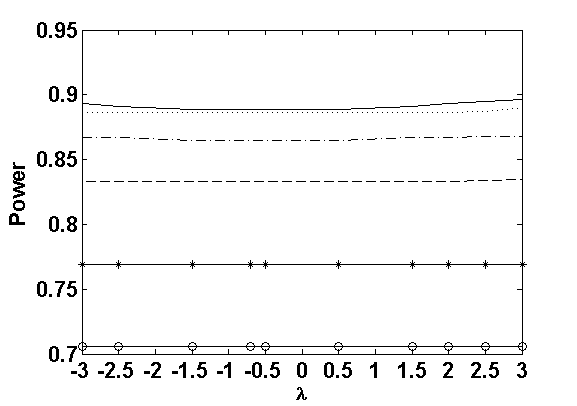}
\label{fig:Spower_300}}
%----------------------------------------------------------------------------------------------
\caption[Plot of approximate power and simulated power of the $S$-divergence based test for normal mean
against contiguous alternatives over the parameter $\lambda$ for different sample sizes $n$ and 
different $\beta=\gamma$]{\footnotesize Plot of approximate power (Appr.) and simulated power (Sim.) of 
the SDT for normal mean against contiguous alternatives over the parameter $\lambda$ for 
different sample sizes $n$
(Solid line: $\beta=\gamma = 0$, Dotted line: $\beta=\gamma = 0.1$, Dash-Dot line: $\beta=\gamma = 0.3$, 
Dashed line: $\beta=\gamma = 0.5$, Marker $*$: $\beta=\gamma = 0.7$, Marker $o$: $\beta=\gamma = 1$).}
 \label{FIG:7Approx_power_comparison}
\end{figure}
%-----------------------------------------------------------------------------------	

To explore how good this approximation is, we also perform a simulation study 
for exactly the same set-up of the normal model as above with 1000 replications 
and present the simulated powers in the same Figure \ref{FIG:7Approx_power_comparison}. 
Interestingly, the simulated powers turn out to be almost the same for different values of $\lambda$ ---
contradicting the corresponding results for approximate power in case of small samples sizes.
This illustrates that the power approximation derived above may not give 
the true picture about the power of $S$-divergence based test for such small sample sizes. 
However, in case of large sample sizes the above power approximation works very well 
at the alternatives considered by us producing power values very close to the simulated power 
for any $\lambda$ and $\beta=\gamma$.
Note that, this fact is not very surprising as the above power approximation is 
only asymptotic in nature.  
Further, the simulated power at any $\gamma=\beta$ values are also very much close to that obtained 
through the expression of the asymptotic power of SDT against contiguous alternative hypotheses in 
Corollary \ref{COR:7cont_power_one}; the corresponding values are presented in 
Table \ref{TAB:7cont_power_norm_mu}. 
In fact, it follows from Corollary \ref{COR:7cont_power_one} that the asymptotic distribution 
of the SDT statistic under contiguous alternative hypotheses $\theta_n$ is the same as that of 
$\zeta_1^{\gamma,\beta} W_{1,\delta}$, where $W_{1,\delta}$ $Z_1$ follows a non-central 
$\chi^2$ distribution with one degrees of freedom and non-centrality parameter $\delta=\Delta^2/\upsilon_\beta$; 
hence its  asymptotic power under contiguous alternatives $\theta_n$ turns out to be 
(see Remark \ref{REM:7SDT0_PIF_1})
$$
P_0 = P(\zeta_1^{\gamma,\beta} W_{1,\delta} > \zeta_1^{\gamma,\beta} \chi_{1,\alpha}^2)
= 1 - G_{\chi_1^2(\delta)}(\chi_{1,\alpha}^2),
$$
where $G_{\chi_1^2(\delta)}$ denotes the distribution function of $W_{1,\delta}$
and $\chi_{1,\alpha}^2$ is the upper $\alpha^{\rm th}$ quantile of the central chi-square distribution 
with one degree of freedom.
Note that the asymptotic power against the contiguous alternative hypotheses are independent of the parameter
$\lambda$ and decreases slightly as the parameter $\gamma=\beta$ increases from $0$ to $1$. 
This shows that there is a loss in efficiency of the SDT with respect to the LRT at the pure data;
however the loss is not as  significant to offset its strong robustness properties. 
%----------------------------------------------------------------------------------------

%----------------------------------------------------------------------------------------
\begin{table}[h]
\centering 
\caption{Contiguous (asymptotic) power of SDT for $\theta_0=0$, $\Delta=\sqrt{10}$. }
\begin{tabular}{ l| r r r r r r } \hline
$\beta=\gamma$	& 0			& 0.1		& 0.3		& 0.5		& 0.7		& 1	\\	\hline
\noalign{\smallskip}
Power 			& 0.89	& 0.88	& 0.86	& 0.83	& 0.79	& 0.72 \\
\noalign{\smallskip}\hline
\end{tabular}
\label{TAB:7cont_power_norm_mu}
\end{table}
%-----------------------------------------------------------------------------------------

Now, we study the extend of robustness of the SDT in the present scenario.
We first consider the asymptotic results proved in Section \ref{SEC:7sample1_simpleTest_robust}, 
where we have seen that  first order influence function of the SDT statistic and any order 
influence function of its level will be zero under any model. Therefore, we consider the 
second order influence function of the SDT statistic; 
under the normal model it can be simplified as 
$$
IF_2(y; T_{\gamma,\lambda}^{(1)}, F_{\theta_0}) 
= \frac{(1+\beta)^{3/2}}{(\sqrt{2\pi}\sigma)^{\beta+\gamma}\sigma^4\sqrt{1+\gamma}} 
(y -\theta_0)^2 e^{-\frac{\beta(y-\theta_0)^2}{\sigma^2}}.
$$
Note that the above influence function is bounded whenever $\beta>0$ and unbounded at $\beta=0$.
Thus the SDT statistic should be always robust if we use a robust MDPDE with $\beta>0$; 
but the use of the non-robust MLE will make the test statistic non-robust also.
Further the maximum possible value of the above influence function can be seen to be 
independent of parameter $\lambda$ and decreases as the values of $\gamma=\beta$ increases. 
Thus, the extent of the robustness of the SDT increases with the values of $\gamma=\beta$ 
but remains unaffected as a function of $\lambda$. Figure \ref{FIG:7IF_normal_mu}\subref{fig:IF_test}
show the influence function of SDT for $\theta_0=0$ and $\sigma=1$.

\begin{figure}[h]
	\centering
	%------------------------------------------------------
	\subfloat[IF of Test Statistics]{
		\includegraphics[width=0.49\textwidth, height=0.3\textwidth] {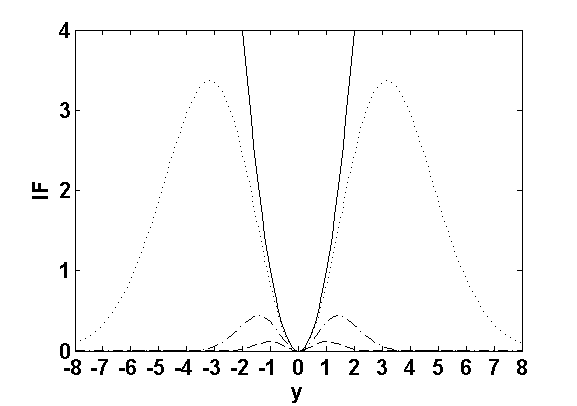}
		\label{fig:IF_test}}
	~ %--------------------------------------------------------------------
	\subfloat[Power IF]{
		\includegraphics[width=0.49\textwidth,  height=0.3\textwidth] {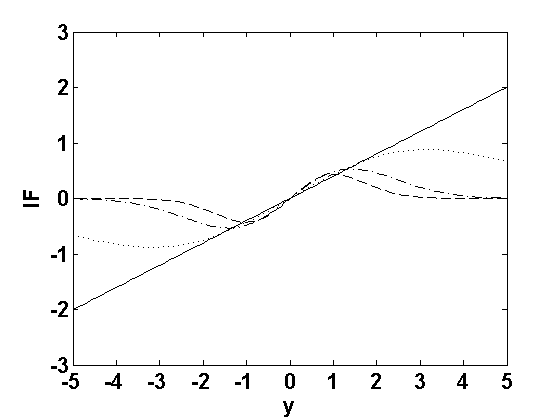}
		\label{fig:IF_power}}
	%--------------------------------------------------------------------
	\caption{Influence functions of the $S$-divergence based test for normal mean with known $\sigma$ over
		different values of $\beta=\gamma$ 
		(Solid line: $\beta=\gamma = 0$, Dotted line: $\beta=\gamma = 0.1$, Dash-Dot line: $\beta=\gamma = 0.5$, 
		Dashed line: $\beta=\gamma = 1$).}
	\label{FIG:7IF_normal_mu}
\end{figure}

 We can also compute the power influence function of the SDT 
 numerically for the normal model; one particular case with $\theta_0=0$, $\sigma=1$ and $\Delta=1$
 is shown in Figure \ref{FIG:7IF_normal_mu}\subref{fig:IF_power} for the $5\%$ level of significance. 
From the robustness perspective, the inference derived from this power influence function 
is again similar to that from the IF of the SDT statistic in Figure \ref{FIG:7IF_normal_mu}\subref{fig:IF_test}.

Finally to see the small sample robustness properties of the SDT and 
compare them with the nature of the IFs derived above,  
we have undertaken a simulation study based on $1000$ replications
with contaminated samples of several types. For the simulation purpose, 
we have assumed $\theta_0=0$ and $\sigma=1$. The empirical size and power of SDT 
under some interesting contamination scenarios are presented in 
Figures \ref{FIG:7cont_size_normal_mu}--\ref{FIG:7contamination_power_exact_normal_mu}.

In Figure \ref{FIG:7cont_size_normal_mu} we have presented the empirical size of the SDT 
for sample size $n=50$ at three different contamination levels ($\epsilon=0, 0.05, 0.10$)
when the data generating distribution is $(1-\epsilon)N(0,1) + \epsilon N(1,1)$.
To give a description of the full range of the observed sizes 
which may not be fully discernible from the figure, we present the actual values 
of the sizes at some specific ($\lambda, \gamma=\beta$) combinations in Table \ref{TAB:7sizes}.   
Clearly, it follows that the size of the SDT remains more stable 
and closer to the nominal level of 5\% 
for the larger values of $\gamma=\beta$ compared to its smaller values like 0 or 0.1. 
For small $\gamma=\beta$ the empirical size of the SDT 
is seen to somewhat depend on the values of $\lambda$;
the  $\lambda$ values closer to zero yields relatively more stable size
closer to the nominal level of the test.
Indeed, there are many members of the class of SDT 
that give more robust size compared to the DPD based test 
for testing the simple hypothesis under consideration.

\begin{figure}[h]
	\centering
	%------------------------------------------------------
	\subfloat[$0\%$ Cont.]{
		\includegraphics[width=0.5\textwidth] {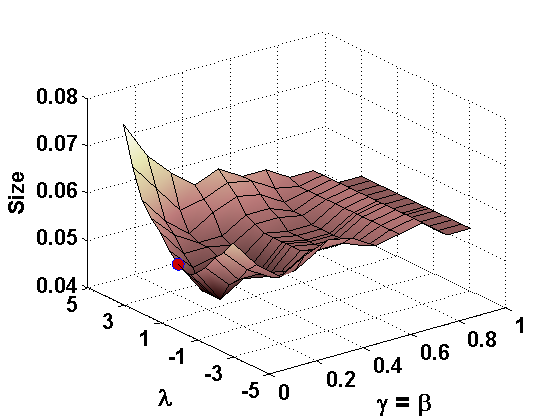}
		\label{fig:size_50_0}}
	\\ %--------------------------------------------------------------------
	\subfloat[$5\%$ Cont.]{
		\includegraphics[width=0.5\textwidth] {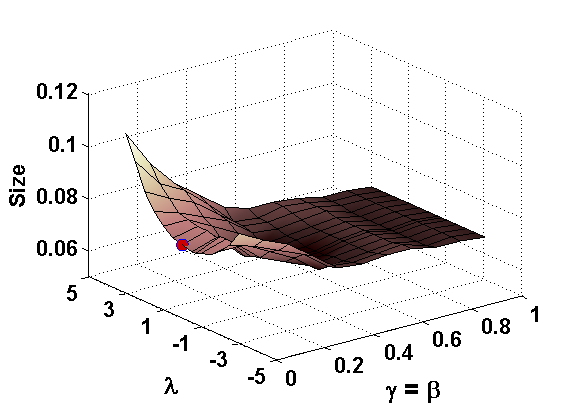}
		\label{fig:size_50_05}}
	~ %--------------------------------------------------------------------
	\subfloat[$10\%$ Cont.]{
		\includegraphics[width=0.5\textwidth] {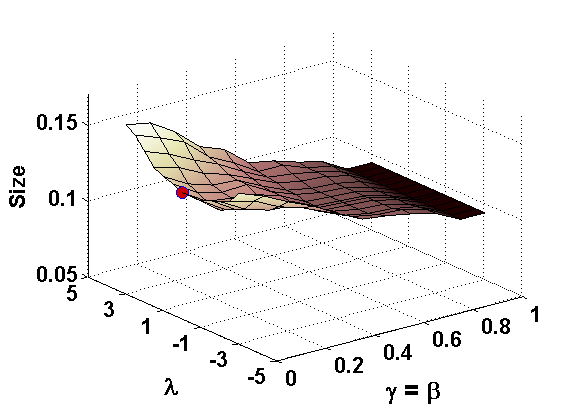}
		\label{fig:size_50_1}}
	%--------------------------------------------------------------------------
	\caption{Empirical size of the SDT for sample size $n=50$ and different contamination proportions; the contaminating distribution is $N(1,1)$}
	\label{FIG:7cont_size_normal_mu}
\end{figure}

%----------------------------------------------------------------------------------------
\begin{table}[h]
	\centering 
	\caption{Some particular values of the empirical size of the SDT as plotted in Figure \ref{FIG:7cont_size_normal_mu}}
	\begin{tabular}{ l l| r r r r r r r} \hline \noalign{\smallskip}
		&	& \multicolumn{7}{c}{$\lambda$}	\\	\hline
		\noalign{\smallskip}
		$\gamma=\beta$	&	$\epsilon$	&	$-3$	&	$-1$	&	$-0.5$	&	0	&	0.5	&	1	&	3	\\	\hline
		\noalign{\smallskip}
		0	&	0	&	0.064	&	0.054	&	0.053	&	0.054	&	0.056	&	0.059	&	0.078	\\	
		&	0.05	&	0.092	&	0.078	&	0.077	&	0.078	&	0.079	&	0.082	&	0.111	\\	
		&	0.10	&	0.149	&	0.133	&	0.133	&	0.133	&	0.136	&	0.138	&	0.161	\\	\hline\noalign{\smallskip}
		0.1	&	0.00	&	0.058	&	0.047	&	0.047	&	0.047	&	0.051	&	0.053	&	0.070	\\	
		&	0.05	&	0.087	&	0.081	&	0.080	&	0.081	&	0.081	&	0.082	&	0.097	\\	
		&	0.1	&	0.143	&	0.123	&	0.121	&	0.123	&	0.126	&	0.129	&	0.158	\\	\hline\noalign{\smallskip}
		0.3	&	0	&	0.057	&	0.049	&	0.049	&	0.049	&	0.050	&	0.052	&	0.061	\\	
		&	0.05	&	0.076	&	0.069	&	0.069	&	0.069	&	0.069	&	0.072	&	0.080	\\	
		&	0.10	&	0.128	&	0.119	&	0.119	&	0.119	&	0.123	&	0.123	&	0.134	\\	\hline\noalign{\smallskip}
		0.5	&	0	&	0.058	&	0.053	&	0.053	&	0.053	&	0.053	&	0.055	&	0.060	\\	
		&	0.05	&	0.068	&	0.065	&	0.065	&	0.065	&	0.066	&	0.066	&	0.070	\\	
		&	0.10	&	0.111	&	0.106	&	0.106	&	0.106	&	0.108	&	0.108	&	0.115	\\	\hline\noalign{\smallskip}
		1	&	0	&	0.053	&	0.053	&	0.053	&	0.053	&	0.053	&	0.053	&	0.053	\\	
		&	0.05	&	0.066	&	0.066	&	0.066	&	0.066	&	0.066	&	0.066	&	0.066	\\	
		&	0.10	&	0.094	&	0.094	&	0.094	&	0.094	&	0.094	&	0.094	&	0.094	\\	
		\noalign{\smallskip}\hline
	\end{tabular}
	\label{TAB:7sizes}
\end{table}
%-----------------------------------------------------------------------------------------

In Figure \ref{FIG:7contamination_power_contig_normal_mu}, 
the power of the SDTs are presented for the contiguous alternative hypotheses
$H_{1n} : \theta = \frac{\Delta}{\sqrt{n}}$ with $\delta=\sqrt{10}$ under 10\% contamination.
In Figures \ref{fig:power_cont_30_m1_1} and \ref{fig:power_cont_50_m1_1},
the true distribution is $0.9 N(\frac{\Delta}{\sqrt{n}}, 1) + 0.1 N(-1,1)$
and the power are calculated at sample sizes $30$ and 50 respectively.
On the other hand, in Figures \ref{fig:power_cont_30_m2_1} and \ref{fig:power_cont_50_m2_1}
the true distribution is  $0.9 N(\frac{\Delta}{\sqrt{n}}, 1) + 0.1 N(-2,1)$
and the powers are again computed at sample sizes 30 and 50 respectively.
The powers are practically constant over $\lambda$, 
although there is some slight variation over $\lambda$ for small $\gamma=\beta$ ($\leq 0.2$).
Also, the SDT with larger values of $\gamma=\beta$ has quite high and stable power
for any contamination scenarios compared to the LRT and 
other SDT with smaller $\gamma=\beta$ values.

Note that, although we have presented the results for only one particular value of $\Delta=\sqrt{10}$,
we have done the power calculation for several other values of $\Delta$ ranging from $\sqrt{8}$ to $\sqrt{12}$.
In general, the overall representation is similar to 
Figure \ref{FIG:7contamination_power_contig_normal_mu},  
but the actual power increases slightly with $\Delta$.

In Figure \ref{FIG:7contamination_power_exact_normal_mu}, we present the power of 
the SDT against the fixed alternative $H_1: \theta =1$ for sample sizes $30$ (in the first column)
and $50$ (in the second column) under $20\%$ contamination. 
We have considered three contamination distributions (in three rows), 
namely $N(-2, 1)$, $N(-3, 1)$ and $N(-4, 1)$. Clearly, for all the contamination scenarios 
the power of the SDTs against the fixed alternative are almost one for $\gamma=\beta\leq 0.3$.
However, the power of the SDTs with smaller values of $\beta=\gamma$ (including LRT)
decreases from around 0.72 to 0.37 as the contamination distribution (with means $-2$ to $-4$)
moves away from the true distribution ($\theta=1$).

\begin{figure}[h]
\centering
%------------------------------------------------------
\subfloat[$(30, N(-1,1))$]{
\includegraphics[width=0.5\textwidth] {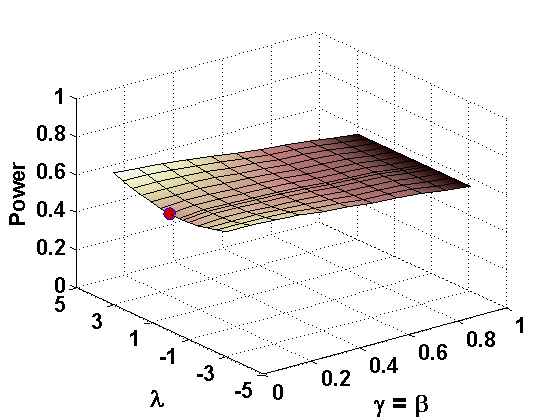}
\label{fig:power_cont_30_m1_1}}
~ %--------------------------------------------------------------------
\subfloat[$(30, N(-2,1))$]{
\includegraphics[width=0.5\textwidth] {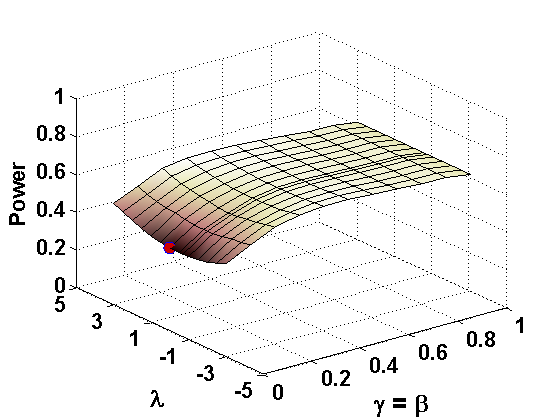}
\label{fig:power_cont_30_m2_1}}
\\ %--------------------------------------------------------------------
\subfloat[$(50, N(-1,1))$]{
\includegraphics[width=0.5\textwidth] {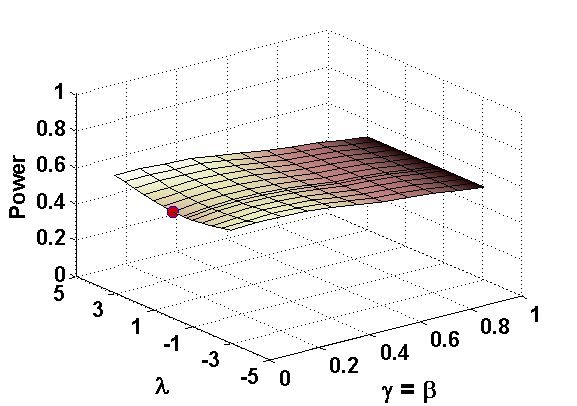}
\label{fig:power_cont_50_m1_1}}
~ %--------------------------------------------------------------------
\subfloat[$(50, N(-2,1))$]{
\includegraphics[width=0.5\textwidth] {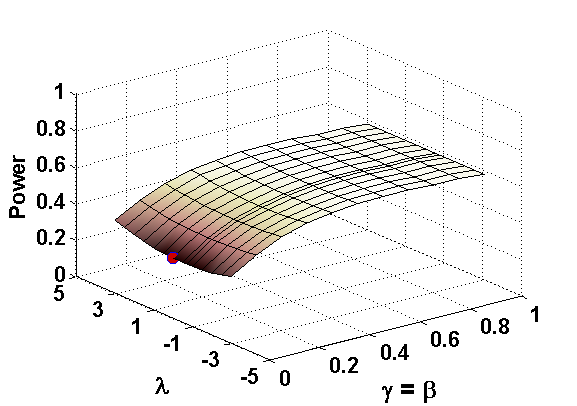}
\label{fig:power_cont_50_m2_1}}
%---------------------------------------------------------------------------
\caption{Empirical power of the SDT at the contiguous alternatives with $\Delta=\sqrt{10}$ 
under $10\%$ contaminations for different combination of (sample size, contamination distribution).}
 \label{FIG:7contamination_power_contig_normal_mu}
\end{figure}

Thus, the proposed SDT with larger $\gamma=\beta$ gives us very useful alternatives to the LRT. 
They have more stable size and contiguous power than the latter test and 
also have satisfactory power against any fixed alternative under contamination. 
We will consolidate these empirical findings along with necessary theoretical results 
to suggest a suitable range of the tuning parameters $\lambda$ and $\gamma=\beta$
for practical usage in Section \ref{SEC:choice_tuning}.

\begin{figure}[!th]
	\centering
	%------------------------------------------------------
	\subfloat[$(30, N(-2,1))$]{
		\includegraphics[width=0.5\textwidth, height=0.27\textwidth] {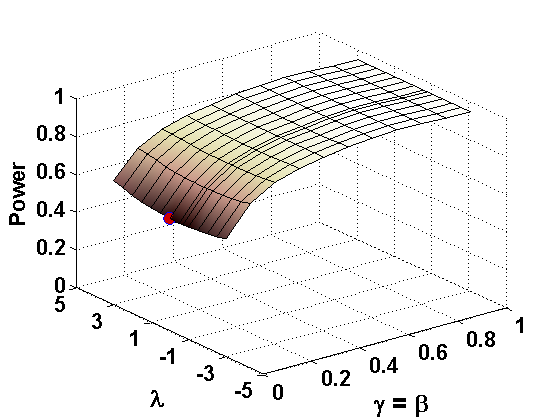}
		\label{fig:power_fixed_30_m2_2}}
	~ %--------------------------------------------------------------------
	\subfloat[$(50, N(-2,1))$]{
		\includegraphics[width=0.5\textwidth, height=0.27\textwidth] {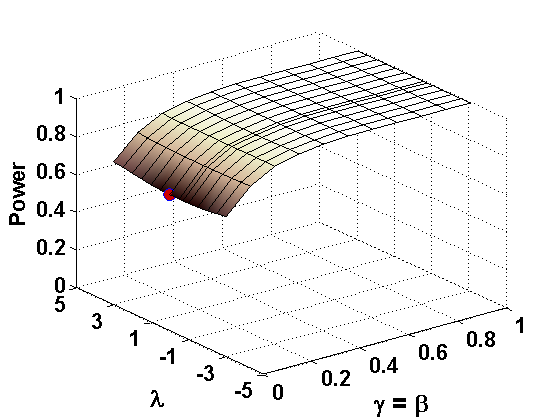}
		\label{fig:power_fixed_50_m2_2}}
	\\ %--------------------------------------------------------------------
	\subfloat[$(30, N(-3,1))$]{
		\includegraphics[width=0.5\textwidth, height=0.27\textwidth] {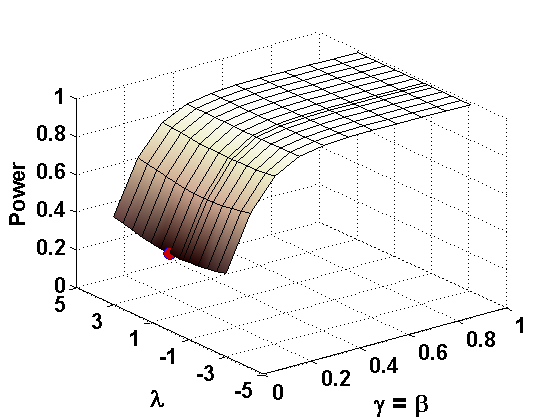}
		\label{fig:power_fixed_30_m3_2}}
	~ %--------------------------------------------------------------------
	\subfloat[$(50, N(-3,1))$]{
		\includegraphics[width=0.5\textwidth, height=0.27\textwidth] {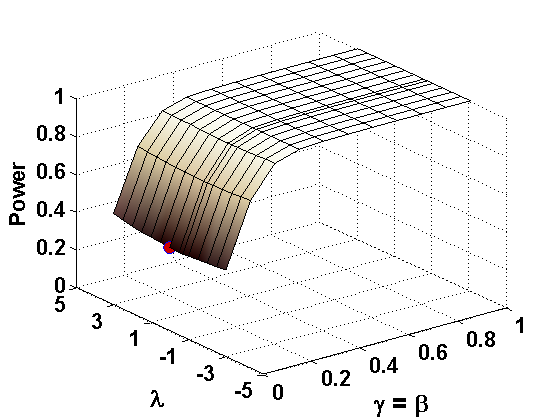}
		\label{fig:power_fixed_50_m3_2}}
	\\ %--------------------------------------------------------------------
	\subfloat[$(30, N(-4,1))$]{
		\includegraphics[width=0.5\textwidth, height=0.27\textwidth] {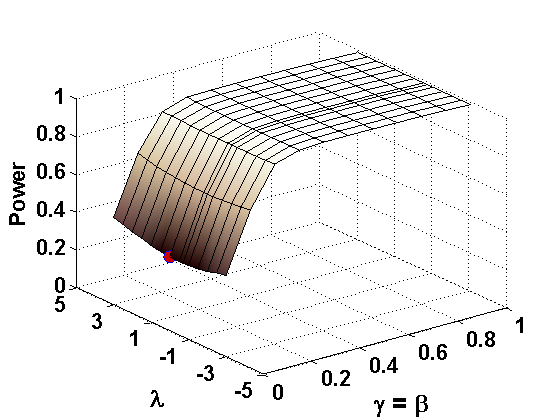}
		\label{fig:power_fixed_30_m4_2}}
	~%--------------------------------------------------------------------------
	\subfloat[$(50, N(-4,1))$]{
		\includegraphics[width=0.5\textwidth, height=0.27\textwidth] {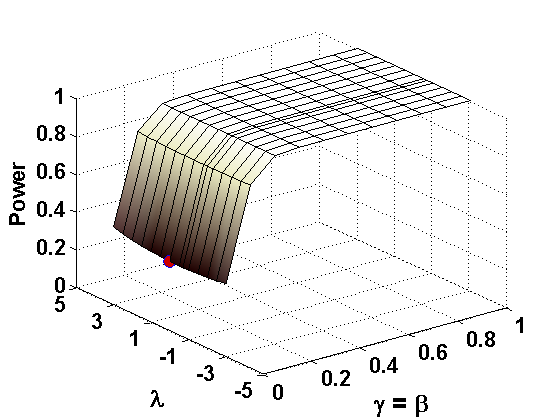}
		\label{fig:power_fixed_50_m4_2}}
	%--------------------------------------------------------------------------
	\caption{Empirical power of the SDT against the fixed alternative  $\theta = 1$ 
		under (heavy) $20\%$ contaminations for different combination of (sample size, contamination distribution).}
	\label{FIG:7contamination_power_exact_normal_mu}
\end{figure}

%----------------------------------------------------------------------------------------
\begin{table}[h]
\centering 
\caption{Values of $c(g_\epsilon)/c(f_{\theta_0})$ for different $\epsilon$ and $\sigma$ under normal model with $\theta_0=0$ }
\begin{tabular}{ll| r r r r r r }  \hline
& & \multicolumn{6}{c}{$\beta=\gamma$}	\\\hline
$\sigma$ & $\epsilon$  & 0		& 0.1		& 0.3		& 0.5		& 0.7		& 1	\\	\hline\hline
0.5 &	0.0005	&	1.0315	&	1.0105	&	1.0029	&	1.0013	&	1.0007	&	1.0006	\\
&	0.001	&	1.0629	&	1.0211	&	1.0057	&	1.0025	&	1.0014	&	1.0010	\\
&	0.005	&	1.3134	&	1.1082	&	1.0290	&	1.0126	&	1.0073	&	1.0054	\\
&	0.01	&	1.6236	&	1.2231	&	1.0593	&	1.0256	&	1.0147	&	1.0107	\\
&	0.02	&	2.2344	&	1.4752	&	1.1242	&	1.0524	&	1.0297	&	1.0216	\\
&	0.05	&	3.9900	&	2.4609	&	1.3592	&	1.1411	&	1.0777	&	1.0560	\\
&	0.1	&	6.6600	&	5.4663	&	1.9592	&	1.3223	&	1.1679	&	1.1186	\\
\hline
1 & 0.0005	&	1.0075	&	1.0030	&	1.0011	&	1.0006	&	1.0006	&	1.0005	\\
& 0.001	&	1.0150	&	1.0059	&	1.0022	&	1.0015	&	1.0011	&	1.0012	\\
& 0.005	&	1.0746	&	1.0296	&	1.0115	&	1.0075	&	1.0058	&	1.0053	\\
& 0.01	&	1.1484	&	1.0597	&	1.0232	&	1.0151	&	1.0116	&	1.0104	\\
& 0.02	&	1.2936	&	1.1213	&	1.0476	&	1.0305	&	1.0238	&	1.0210	\\
& 0.05	&	1.7100	&	1.3190	&	1.1265	&	1.0799	&	1.0619	&	1.0543	\\
& 0.1		&	2.3400	&	1.6984	&	1.2821	&	1.1734	&	1.1320	&	1.1147 \\\hline
2&	0.0005	&	1.0015	&	1.0007	&	1.0005	&	1.0007	&	1.0006	&	1.0005	\\
&	0.001	&	1.0030	&	1.0015	&	1.0012	&	1.0011	&	1.0009	&	1.0009	\\
&	0.005	&	1.0149	&	1.0078	&	1.0058	&	1.0055	&	1.0052	&	1.0051	\\
&	0.01	&	1.0296	&	1.0156	&	1.0118	&	1.0113	&	1.0107	&	1.0101	\\
&	0.02	&	1.0584	&	1.0312	&	1.0239	&	1.0225	&	1.0213	&	1.0203	\\
&	0.05	&	1.1400	&	1.0786	&	1.0617	&	1.0584	&	1.0552	&	1.0529	\\
&	0.1	&	1.2600	&	1.1591	&	1.1307	&	1.1239	&	1.1171	&	1.1123	\\
\hline
\end{tabular}
\label{TAB:7chi-sq_infl}
\end{table}
%-----------------------------------------------------------------------------------------

Finally, we complete this section with an illustration for the practical implication of 
our third robustness measure, the chi-square inflation factor 
(derived in Subsection \ref{SEC:Chi-square-inflation}), studied here under the normal model. 
As the null mean is $\theta_0 =0$ and the variance $\sigma^2$ is known, 
we can easily compute the values of the the chi-square inflation factor $c(g_\epsilon)$ 
under the contaminated model $g_\epsilon$ as defined in Subsection \ref{SEC:Chi-square-inflation}.
However, note that the chi-square inflation factor of the SDT is independent of 
the parameter $\lambda$ as shown in Theorem \ref{THM:7Chi_infl_fact_slope}.
Table \ref{TAB:7chi-sq_infl} presents the values of $c(g_\epsilon)/c(f_{\theta_0})$ 
for different  contamination proportions $\epsilon$ at the point $y=4$
%a point outside the $3$-$\sigma$ limit with $\sigma=1$, 
for various values of $\sigma$ and $\beta=\gamma$.
As mentioned earlier, larger the value of $c(g_\epsilon)/c(f_{\theta_0})$ relative to $1$, 
more liberal will the confidence interval based on the SDT be; $c(g_\epsilon)/c(f_{\theta_0})=1$ 
generates the confidence interval with the correct level of significance. 
Therefore, it follows from the Table \ref{TAB:7chi-sq_infl} that the SDT with $\beta=\gamma=0$ 
is highly non-robust producing extremely liberal confidence intervals even at small contamination 
proportions like $\epsilon=0.01$. However, the SDTs with larger $\gamma=\beta$ 
remain stable even under higher contaminations like $\epsilon=0.05$ or $0.1$.
The chi-square inflation factor also depends on the parameters $y$ and $\sigma$;
note that $\theta_0$ is fixed by the null hypothesis. 
Also, it is clear from the theory discussed in Section \ref{SEC:Chi-square-inflation}
that the effect of the contamination increases as its mass moves away from the 
true distribution, i.e., as the values of $y$ increases in magnitude for the present case.
To examine the effect of $\sigma$, the results for 3 particular values of $\sigma$ are 
presented  in Table \ref{TAB:7chi-sq_infl}. 
Clearly, for any fixed contamination proportion $\epsilon$ at the point  $y$, 
the SDT generates more accurate confidence interval as the value of $\sigma$ increases.
This fact is quite intuitive because the effective distance between the contamination distribution 
and the true distribution decreases as $\sigma$ increases and 
hence the effect of contamination reduces.

Further, for the normal model the slope of the chi-square inflation factor 
at the infinitesimal contamination can be computed easily using 
Theorem \ref{THM:7Chi_infl_fact_slope} and has the simplified form given by 
$$
\frac{\partial}{\partial \epsilon}c(g_\epsilon) \big|_{\epsilon=0} 
= \frac{\kappa_\gamma (1+\beta)^2\sqrt{1+2\beta}}{\kappa_\beta^2 \kappa_{2\beta}} 
(y-\theta_0) e^{-\frac{\beta(y-\theta_0)^2}{\sigma^2}}.
$$
Again, it is clear from the above expression that 
the slope of the chi-square inflation factor remains bounded 
with respect to the contamination point for all $\beta>0$ 
implying the stability of the corresponding SDT.
On the other hand, at $\beta=0$ it becomes unbounded in $y$ 
which further illustrates the non-robust nature of the corresponding SDT as observed above.

%========================================================================
\section{On the Choice of the Tuning Parameters}
\label{SEC:choice_tuning}

In previous sections, we have illustrated the performance of the proposed SDT both theoretically and empirically
for different values of the tuning parameters $\beta = \gamma$ and $\lambda$. 
So, some guidance about the actual choice of the divergence within 
the large class of available members of the $S$-divergence family are necessary here.

The theoretical robustness of the proposed SDT is seen to depend solely 
on the MDPDE used through the parameter $\beta$;
as $\beta$ gets larger, the robustness increases. However, the empirical power and size 
under contamination are seen to depend on $\lambda$
at smaller values of $\gamma=\beta\leq 0.2$ though their actual values are unstable 
for such smaller values of $\gamma=\beta$.
On the other hand, for larger values of $\gamma=\beta\geq 0.3$, both the empirical power and size  are quite stable 
and almost independent of $\lambda$; in this region the powers are found to be satisfactory high values and 
sizes are quite close to their nominal levels. However, a large value of $\beta=\gamma$ leads to a loss in 
the power of test under pure data as we have seen in Table \ref{TAB:7cont_power_norm_mu}; 
but the loss is not very significant even at $\beta=\gamma=0.5$.  
Besides the SDT does not necessarily have satisfactory sizes for all values of $\lambda$ --
the observed sizes are generally close to nominal ones only for $\lambda$ values near zero of the negative side.
%Further, the SDT with all values of $\lambda$ cannot give us the nominal level always under pure data
%-- they yield the same only for $\lambda$ values near to zero on the negative side. 

Therefore, our empirical calculations, along with our theoretical findings,  
indicate that the preferable region of tuning parameter combination 
would approximately be the rectangle $\gamma=\beta \in [0.3, 0.5]$ and $\lambda \in [-0.5, 0]$. 
Empirically, we have observed that this is the preferable region in the sense that 
they maintain levels close to the nominal level and exhibit reasonably high powers. 
Tests that cannot maintain the nominal level are usually of little practical value 
as one does not know whether a higher observed power is an actual phenomenon 
or the consequence of an unstable size.
So, we do not advocate the use of such tests which have highly variable sizes in small sample.

%========================================================================
\section{Possible Extension: Composite Hypothesis Testing}
\label{SEC:composite}

Although we have focused only on simple hypothesis testing throughout the present paper,
it is worth noting that all the concepts and results derived for the simple null can be generalized 
for composite hypothesis testing.  
This would indicate a wider scope for the proposals in this paper and their significance. 
Consider the set-up presented in the previous sections and 
let $\Theta_0$ be a proper subset of the parameter space $\Theta$.
Then, a composite hypothesis is given by 
\begin{equation}
H_0 : \theta \in \Theta_0 ~~~ \mbox{against} ~~~~ H_1 : \theta \notin \Theta_0.
\label{EQ:7composite_hyp}
\end{equation}
Note that, whenever $\Theta=\{\theta_0\}$ contains only one element $\theta_0$ then the above composite hypothesis 
(\ref{EQ:7composite_hyp}) coincides with the simple null hypothesis (\ref{EQ:7simple_hyp}).

Using a similar idea as in Section \ref{SEC:7sample1_simpleTest}, 
we can also construct a general family of test statistics 
based on the $S$-divergence measures. Suppose that $\hat{\theta}_\beta$ denote the unrestricted MDPDE of $\theta$ 
and $\widetilde{\theta}_\beta$ denote the restricted MDPDE obtained by minimizing 
the DPD with tuning parameter $\beta$ 
between the data and the model over the restricted subspace $\Theta_0$. 
The asymptotic distribution and influence functions of the restricted MDPDE 
can be found in \cite{Basu/etc:2013b} and \cite{Ghosh:2014}. 
Therefore, the general $S$-divergence based test for testing 
the composite hypothesis  (\ref{EQ:7composite_hyp}) is given  by
\begin{eqnarray}
\widetilde{\xi_n^{\gamma,\lambda}}({\hat{\theta}_\beta}, {\widetilde{\theta}_\beta}) 
= 2 n S_{(\gamma,\lambda)}(f_{\hat{\theta}_\beta}, f_{\widetilde{\theta}_\beta}).
\label{EQ:SDT_composite}
\end{eqnarray}
Note that if $\Theta_0=\{\theta_0\}$ then $\widetilde{\theta}_\beta=\theta_0$ for each $\beta$ and 
this general test statistic (\ref{EQ:SDT_composite}) coincides with our initial test statistics .

%Here also we can replace the MDPDE by suitable minimum $S$-divergence estimators,
%although we are suggesting the MDPDE just to avoid complications related to non-parametric smoothing.
All the results including the asymptotic distributions and  robustness properties 
can easily be extended to the class of test statistics (\ref{EQ:SDT_composite}); 
the advantages are similar to the case of the simple null except for a few minor changes 
in the numerical values and constants due to the restrictions imposed by the null hypothesis. 
However, considering the length of the present paper, we have decided 
to consider the details of this extension in a future paper.
%not to present those extensions in the current paper; those will be made available in a future article. 

%========================================================================
\section{Concluding Remarks}
\label{SEC:conclusion}

The excellent robustness properties of the tests of parametric hypotheses proposed 
by \cite{Basu/etc:2013a} have already been empirically observed and heuristically argued for.
However, the literature remains incomplete unless the theoretical robustness properties 
of these tests are properly established.
These properties have been carefully assembled in this paper and 
the different theoretical properties of this testing procedure are now established. 
The results theoretically conform the empirical observations of \cite{Basu/etc:2013a} and 
make the credentials of the density power divergence test more complete.

\appendix
\section{Conditions for Asymptotic Derivations}\label{App:conditions}

We assume that the true density $g$ belongs to the model family with true parameter value $\theta_0$,
i.e., $g=f_{\theta_0}$ and state the necessary conditions under this set-up.

\bigskip
\noindent
\textbf{Lehman Conditions \cite[][p.~429]{Lehmann:1983}:}

\begin{itemize}
	\item[(A)] There is an open subset of $\omega$ of the parameter space 	$\Theta$, 
	containing the true parameter value $\theta_0$ such that for almost all $x \in {\cal X}$,
	and all $\theta \in \omega$, the density $f_\theta(x)$ is three times differentiable with respect to $\theta$.

	\item[(B)] The first and second logarithmic derivatives of $f_\theta$ satisfy the equations 
	$$
	E_\theta\left[\nabla \log f_\theta(X)\right] = 0,
	$$
	and
	$$
	I(\theta) = E_\theta\left[(\nabla \log f_\theta(X))(\nabla \log f_\theta(X))^T\right]
	= E_\theta\left[-\nabla^2 \log f_\theta(X)\right].
	$$
	
	\item[(C)] The matrix $I(\theta)$ is positive definite with  all entries finite for all $\theta \in \omega$,
	and hence the components  $(\nabla \log f_\theta(X))$ are affinely independent with probability one. 
	
	\item[(D)] For all $j, k, l$, there exists functions $M_{jkl}$ with finite expectation 
	(under true distribution) such that 
	$$
	\left|\nabla_{jkl} \log f_\theta(x)\right| \leq M_{jkl}(x),~~~~~~~\mbox{ for all}~~\theta \in \omega.
	$$ 
\end{itemize}
%-------------

\noindent
%These are standard conditions for any asymptotic derivation of the distribution of parameter estimators. 
The above conditions are standard for establishing the asymptotic distributions of estimators 
in many parametric situations, and the relevance of these conditions are well known.

%===============================================================================
\bigskip\bigskip
\noindent
\textbf{Basu et al.~Conditions \cite[][p.~304]{Basu/etc:2011} at the model:}

\begin{itemize}
\item[(D1)] The support of the distribution function $F_\theta$, i.e.,  
the set ${\cal X} = \{x|f_\theta(x) > 0\}$ is independent of $\theta$. 
	
\item[(D2)] There is an open subset of $\omega$ of the parameter space 	$\Theta$, 
containing the true parameter value $\theta_0$ such that for almost all $x \in {\cal X}$,
and all $\theta \in \omega$, the density 	$f_\theta(x)$ is three times differentiable with respect to $\theta$
and the third partial derivatives are continuous with respect to	$\theta$.
	
\item[(D3)] The integral $\int f_\theta^{1+\beta}(x) dx$   
can be differentiated three times with respect to $\theta$, and the 
derivatives can be taken under the integral sign.
	
\item[(D4)] The ${p \times p}$ matrix $J_\beta(\theta)$, 
defined in Section \ref{SEC:7sample1_simpleTest},  is positive definite.
	
\item[(D5)] There exist functions $M_{jkl}(x)$ with finite expectation under the true distribution such that
$$
\left|\nabla_{jkl} d_{\beta}(\delta_x,f_\theta) \right| \leq M_{jkl}(x)
~{\rm for~all}~\theta \in \omega,~{\rm for~all}~j,k,l,
$$ 
where $\delta_x$ is the density function of the degenerate distribution at $x$.
\end{itemize}
%-------------

\noindent
%\textit{Implications and Significances:}
The first two conditions (D1)--(D2) relate to our parametric assumptions, 
which are routine in asymptotic derivations and are satisfied by most parametric models. 
The last three conditions depend both on the structure of the density power divergence 
as well as the value of $\beta$ and are necessary to establish the required asymptotics. 
It is not difficult to verify that these conditions are satisfied by common parametric models 
like the normal, exponential, Poisson, geometric etc. for all $\beta \geq 0$.
%The first two conditions (D1)--(D2) are solely depends on the assumed model 
%and are satisfied by most common statistical models.
%However the next three conditions depend also on the value of $\beta$ 
%and the structure of the density power divergence.
%It has been verified that they holds for common parametric models 
%like normal, exponential, Poisson,  Geometric etc. for any $\beta \geq 0$.
%

%===============================================================================
\bigskip\bigskip
\noindent\textbf{Acknowledgement:}\\
This research is partially supported by Grant MTM 2012-331-40. 
The authors also gratefully acknowledge the comments of a referee 
which led to an improved version of the manuscript.

%========================================================================


\begin{thebibliography}{}

\bibitem[\protect\citeauthoryear{Basu, Harris, Hjort, and Jones}{Basu et~al.}{1998}]{Basu/etc:1998}
Basu, A., I.~R. Harris, N.~L. Hjort, and M.~C. Jones (1998).
\newblock Robust and efficient estimation by minimising a density power
  divergence.
\newblock {\em Biometrika\/}~{\em 85}, 549--559.


\bibitem[\protect\citeauthoryear{Basu, Shioya and Park}{Basu et~al.}{2011}]{Basu/etc:2011}
Basu, A., Shioya, H. and Park, C. (2011).
\newblock {\em Statistical Inference: The Minimum Distance Approach}.
\newblock Chapman \& Hall/CRC. 

\bibitem[\protect\citeauthoryear{Basu, Mandal, Martin, and Pardo}{Basu  et~al.}{2013a}]{Basu/etc:2013a}
Basu, A., Mandal, A., Martin, N., and Pardo, L.~(2013a).
\newblock Testing statistical hypotheses based on the density power divergence.
\newblock {\em Annals of the Institute of Statistical Mathematics\/}~{\em 65},
  319--348.
  
\bibitem[\protect\citeauthoryear{Basu, Mandal, Martin, and Pardo}{Basu et~al.}{2013b}]{Basu/etc:2013b}
Basu, A., Mandal, A., Martin, N., and Pardo, L.~(2013b).
\newblock Density Power Divergence Tests for Composite Null Hypotheses.
\newblock {\em ArXiv Pre-Print\/},~{\em arXiv:1403.0330 [stat.ME]}.  


\bibitem[\protect\citeauthoryear{Cressie and Read}{Cressie and Read}{1984}]{Cressie/Read:1984}
Cressie, N. and T. R. C. Read (1984). 
\newblock Multinomial goodness-of-fit tests. 
\newblock  {\em Journal of Royal Statistical Society - B\/}~{\em 46}, 440--464.


\bibitem[\protect\citeauthoryear{Fisher}{Fisher}{1925}]{Fisher:1925b}
Fisher, R.~A. (1925) (10th Ed., 1946).
\newblock {\em  Statistical Methods for Research Workers}.
\newblock Edinburgh, Olyber \& Boyd. 
  
  
\bibitem[\protect\citeauthoryear{Fisher}{Fisher}{1935}]{Fisher:1935}
Fisher, R.~A. (1935).
\newblock The logic of inductive inference (with discussion).
\newblock {\em Journal of the Royal Statistical Society\/}~{\em 98}, 39--82.


\bibitem[\protect\citeauthoryear{Ghosh}{Ghosh}{2014}]{Ghosh:2014}
Ghosh, A.~(2014). 
\newblock Influence Function of the Restricted Minimum Divergence Estimators : A General Form.
\newblock Pre-print.

\bibitem[\protect\citeauthoryear{Ghosh, Harris, Maji, Basu and Pardo}{Ghosh et~al.}{2013}]{Ghosh/etc:2013a}
Ghosh, A., Harris, I. R., Maji, A., Basu, A., Pardo, L. (2013). 
\newblock A Generalized Divergence for Statistical Inference. 
\newblock {\it Technical Report}, {\bf BIRU/2013/3}, 
Bayesian and Interdisciplinary Research Unit, Indian Statistical Institute, Kolkata, India.




\bibitem[\protect\citeauthoryear{Hampel, Ronchetti, Rousseeuw, and Stahel}{Hampel et~al.}{1986}]{Hampel/etc:1986}
Hampel, F.~R., E.~Ronchetti, P.~J. Rousseeuw, and W.~Stahel (1986).
\newblock {\em Robust Statistics: The Approach Based on Influence Functions}.
\newblock New York, USA: John Wiley \& Sons.



\bibitem[\protect\citeauthoryear{Harville}{Harville}{1971}]{Harville:1971}
Harville, D.~A. (1971). 
\newblock On the distribution of linear combinations of non-central chi-squares. 
\newblock {\em Annals of Mathematical Statistics\/}~{\em 42}, 809-–811.  


\bibitem[\protect\citeauthoryear{Heritier and Ronchetti}{Heritier and Ronchetti}{1994}]{Heritier/Ronchetti:1994}
Heritier, S. and Ronchetti, E. (1994). 
\newblock Robust bounded-influence tests in general parametric models. 
\newblock  {\em Journal of the American Statistical Association\/}~{\em 89}, 897--904.

\bibitem[\protect\citeauthoryear{Huber-Carol}{Huber-Carol}{1970}]{Huber/Carol:1970}
Huber-Carol, C. (1970).
\newblock {\em Etude asymptotique de tests robustes}.
\newblock Ph.\ D. thesis, ETH, Zurich.


\bibitem[\protect\citeauthoryear{Kotz, Johnson, and Boyd}{Kotz et~al.}{1967a}]{Kotz/etc:1967a}
Kotz, S., Johnson, N.~L., and Boyd, D.~W. (1967a).
\newblock Series representations of distributions of quadratic forms in normal variables.~I.~Central case.
\newblock {\em Annals of Mathematical Statistics\/}~{\em 38}, 823--837.


\bibitem[\protect\citeauthoryear{Kotz, Johnson, and Boyd}{Kotz et~al.}{1967b}]{Kotz/etc:1967b}
Kotz, S., Johnson, N.~L., and Boyd, D.~W. (1967b).
\newblock Series representations of distributions of quadratic forms in normal variables.~I.~Non-central case.
\newblock {\em Annals of Mathematical Statistics\/}~{\em 38}, 838--848.


\bibitem[\protect\citeauthoryear{Lehmann}{Lehmann}{1983}]{Lehmann:1983}
Lehmann, E.~L. (1983).
\newblock {\em Theory of Point Estimation}.
\newblock John Wiley \& Sons.
  
  
\bibitem[\protect\citeauthoryear{Lindsay}{Lindsay}{1994}]{Lindsay:1994}
Lindsay, B.~G. (1994).
\newblock Efficiency versus robustness: The case for minimum {H}ellinger
  distance and related methods.
\newblock {\em Annals of Statistics\/}~{\em 22}, 1081--1114.

\bibitem[\protect\citeauthoryear{Liu, Tang and Zhang}{Liu et~al.}{2009}]{Liu/etc:2009}
Liu, H., Tang, Y. and Zhang, H.~H. (2009). 
\newblock General oracle inequalities for gibbs posterior with application to ranking. 
\newblock {\em Computational Statistics and Data analysis\/}~{\em 53}, 853--856.


\bibitem[\protect\citeauthoryear{Neyman and Pearson}{Neyman and  Pearson}{1928}]{Neyman/Pearson:1928}
Neyman, J. and E.~S. Pearson (1928).
\newblock On the use and interpretation of certain test criteria for purposes
  of statistical inference.
\newblock {\em Biometrika\/}~{\em 20A}, 175--240.


\bibitem[\protect\citeauthoryear{Neyman and Pearson}{Neyman and  Pearson}{1933a}]{Neyman/Pearson:1933a}
Neyman, J. and E.~S. Pearson (1933a).
\newblock On the problem of the Most Efficient Tests of Statistical Hypotheses.
\newblock {\em Philosophical Transactions of the Royal Society of London\/}~{\em Series A, 231}, 
289--337.


\bibitem[\protect\citeauthoryear{Neyman and Pearson}{Neyman and Pearson}{1933b}]{Neyman/Pearson:1933b}
Neyman, J. and E.~S. Pearson (1933b).
\newblock The Testing of Statistical Hypotheses in Relation to Probabilities A Priori.
\newblock {\em Proceeding of the Cambridge Philosophical Society\/}~{\em 29}, 492--510.


\bibitem[\protect\citeauthoryear{Press}{Press}{1966}]{Press:1966}
Press, S.~J. (1966). 
\newblock Linear combinations of non-central chi-square variates. 
\newblock {\em Annals of Mathematical Statistics\/}~{\em 37}, 480-–487.


\bibitem[\protect\citeauthoryear{Rousseeuw and Ronchetti}{Rousseeuw and Ronchetti}{1979}]{Rousseeuw/Ronchetti:1979}
Rousseeuw,~P.~J. and Ronchetti,~E. (1979).
\newblock The influence curve for tests. 
\newblock {\em Research Report} {\bf 21}, Fachgruppe f\"{u}r Statistik, ETH, Zurich.


\bibitem[\protect\citeauthoryear{Rousseeuw and Ronchetti}{Rousseeuw and Ronchetti}{1981}]{Rousseeuw/Ronchetti:1981}
Rousseeuw,~P.~J. and Ronchetti,~E. (1981).
\newblock Influence curves for general statistics. 
\newblock {\em Journal of Computational and Applied Mathematics\/}~{\em 7}, 161--166.


\bibitem[\protect\citeauthoryear{Toma and Broniatowski}{Toma and Broniatowski}{2011}]{Toma/Broniatowski:2010}
Toma, A. and M.~Broniatowski (2011).
\newblock Dual divergence estimators and tests: robustness results.
\newblock {\em Journal of Multivariate Analysis\/}~{\em 102}, 20--36.


  \bibitem[\protect\citeauthoryear{Wilks}{Wilks}{1938}]{Wilks:1938}
  Wilks, S.~S. (1938).
  \newblock The large sample distribution of the likelihood ratio for testing
    composite hypothesis.
  \newblock {\em Annals of Mathematical Statistics\/}~{\em 9}, 60--62.

\end{thebibliography}
\end{document}